\documentclass[10pt]{article}

\usepackage{lmodern}
\usepackage[T1]{fontenc}
\usepackage{amsmath}
\usepackage{filecontents}
\usepackage{amsthm}
\usepackage{amssymb}
\usepackage{mathabx} 
\usepackage{url}
\usepackage{latexsym}
\usepackage{titlefoot}
\usepackage[small]{titlesec}
\usepackage[small,it]{caption}
\usepackage{mathdots}
\usepackage{mathtools}
\usepackage{bm}
\usepackage{multirow}
\usepackage{tikz}
\usepackage{multicol}
\usepackage[nocompress]{cite}

\usetikzlibrary{patterns}

\makeatletter

\let\c@table\c@figure
\makeatother 

\usepackage[labelformat=simple,skip=-2pt]{subcaption}

\usepackage{color}
\definecolor{lightgray}{rgb}{0.8, 0.8, 0.8}
\definecolor{darkgray}{rgb}{0.7, 0.7, 0.7}
\definecolor{darkblue}{rgb}{0, 0, .4}
\definecolor{dark-gray}{gray}{0.35}

\usepackage[bookmarks]{hyperref}
\hypersetup{
        colorlinks=true,
        linkcolor=darkblue,
        anchorcolor=darkblue,
        citecolor=darkblue,
        urlcolor=darkblue,
        pdfpagemode=UseThumbs,
        pdftitle={Deflatability of Permutation Classes},
        pdfsubject={Combinatorics},
        pdfauthor={Albert, Atkinson, Homberger, Pantone},
}

\newcounter{todocounter}

\theoremstyle{plain}
\newtheorem{theorem}{Theorem}[section]
\newtheorem{proposition}[theorem]{Proposition}
\newtheorem{lemma}[theorem]{Lemma}

\newtheorem{conjecture}[theorem]{Conjecture}
\newtheorem{question}[theorem]{Question}

\makeatletter
\newfont{\footsc}{cmcsc10 at 8truept}
\newfont{\footbf}{cmbx10 at 8truept}
\newfont{\footrm}{cmr10 at 10truept}
\renewenvironment{abstract}%
                {
                  \begin{list}{}%
                     {\setlength{\rightmargin}{1in}%
                      \setlength{\leftmargin}{1in}}%
                   \item[]\ignorespaces\begin{small}}%
                 {\end{small}\unskip\end{list}}

\newcommand{\Av}{\operatorname{Av}}
\newcommand{\Cl}{\operatorname{Cl}}

\newcommand{\ds}{\displaystyle}

\newcommand{\C}{\mathcal{C}}
\newcommand{\D}{\mathcal{D}}

\newcommand{\SD}{\operatorname{SD}}

\newcommand{\eval}[2][\right]{\relax\ifx#1\right\relax \left.\fi#2#1\rvert}

\datefoot{\today}
\amssubj{05A05}

\newpagestyle{main}[\small]{
        \headrule
        \sethead[\usepage][][]
        {\sc Deflatability of Permutation Classes}{}{\usepage}}

\title{\sc Deflatability of Permutation Classes}
\author{
	M. H. Albert\\
	M. D. Atkinson\footnote{Corresponding author: \url{mike@cs.otago.ac.nz}}\\[-0.25ex]
	\small Department of Computer Science\\[-0.5ex]
	\small University of Otago\\[-0.5ex]
	\small New Zealand	
	\and
	Cheyne Homberger\\
	Jay Pantone\footnote{Pantone's research was sponsored by the National Science Foundation under Grant Number DMS-1301692.}\\[-0.25ex]
	\small Department of Mathematics\\[-0.5ex]
	\small University of Florida\\[-0.5ex]
	\small USA	
}

\titleformat{\section}
        {\large\sc}
        {\thesection.}{1em}{}

\date{}

\begin{document}
\maketitle

\pagestyle{main}

\begin{abstract}
A deflatable permutation class is one in which the simple permutations are contained in a proper subclass. Deflatable permutation classes are often easier to describe and enumerate than non-deflatable ones. Some theorems which guarantee non-deflatability are proved and examples of both deflatable and non-deflatable principal classes are given.

\end{abstract}

\section{Introduction}
\label{section:introduction}

This paper is inspired by a series of recent enumerative and structural results in the theory of permutation pattern classes.  These have all been proven by a common technique which relies on establishing a sparseness property of the simple permutations in a particular permutation class.  When this property holds, much of the analysis of the class can be carried out in a substantially smaller permutation class. This raises the question, interesting in its own right, of characterising when this sparseness property, which we will call \emph{deflatability}, holds. This paper answers that question, at least in part, for principal permutation classes.

In the remainder of this section we define the basic terms of our subject, explain our motivation in concrete terms, and summarise our results. Section~\ref{section:prelim-lemmas} contains preliminary results needed in the subsequent section. Section~\ref{section:notdef} is devoted to proving cases in which deflatability fails, while Section~\ref{section:def} provides example of principal classes which are deflatable. Finally, Section~\ref{section:concl} discusses the remaining unknown cases along with some open questions.

The fundamental concept in permutation class theory is the relation of permutation containment.  A permutation $\alpha$ is contained as a pattern (or subpermutation) in another permutation $\beta$ (denoted $\alpha\leq\beta$) if, when both are written in one line notation, $\beta$ has a subsequence whose terms are ordered in the same relative manner as the terms of $\alpha$. For instance, $312\leq 2531647$ because the entries of the subsequence $514$ follow the same relative order as the permutation $312$.  This relation is more clearly displayed in the diagrams of the permutations where we plot the points $(i, \beta(i))$ of a permutation $\beta$ in the $(x,y)$ plane.  For example, the diagram of $2531647$ is shown in Figure~\ref{figure:perm-ex-1}.

\begin{figure}
	\minipage{0.5\textwidth}
	\begin{center}
		\begin{tikzpicture}[scale=0.3]
			\draw[step=1cm,gray,very thin] (0.0,0.0) grid (8,8);
			\draw (0,0) rectangle (8,8);
			\fill (1,2) circle (0.2cm);
			\fill[color=gray] (2,5) circle (0.2cm);
			\fill (3,3) circle (0.2cm);
			\fill[color=gray] (4,1) circle (0.2cm);
			\fill (5,6) circle (0.2cm);
			\fill[color=gray] (6,4) circle (0.2cm);
			\fill (7,7) circle (0.2cm);
		\end{tikzpicture}
	\end{center}
	\caption{The diagram of the permutation $2531647$. The three gray points represent the pattern $312$ contained in $2531647$.\ \\}
	\label{figure:perm-ex-1}
\endminipage\hfill
\minipage{0.5\textwidth}
	\begin{center}
		\begin{tikzpicture}[scale=0.3]
			\fill [color=lightgray] (0.5,0.5) rectangle (4.5, 4.5);
			\draw[step=1cm,gray,very thin] (0.0,0.0) grid (8,8);
			\draw (0,0) rectangle (8,8);
			\fill (1,4) circle (0.2cm);
			\fill (2,3) circle (0.2cm);
			\fill (3,7) circle (0.2cm);
			\fill (4,1) circle (0.2cm);
			\fill (5,2) circle (0.2cm);
			\fill (6,6) circle (0.2cm);
			\fill (7,5) circle (0.2cm);
		\end{tikzpicture}
	\end{center}
	\caption{The permutation $4371265 = 2413[21,1,12,21]$. The shaded square represents a box which is cut by two cut points, one by position and one by value.}
	\label{figure:perm-ex-2}
	\endminipage
\end{figure}

Such diagrams will be used extensively, and we now define some associated terms.  A \emph{box} in a permutation diagram is a rectangular region  containing a subset of the points.  A {\em cut point} of a box is a point of the permutation  outside the box but whose position is between  the leftmost and rightmost points of the box (cut by position), or whose value is between the maximum and minimum points of the box (cut by value). An \emph{interval} of a permutation is a box which is not cut by position nor by value. Equivalently, an interval in a permutation is a set of entries whose indices and values each form a contiguous set, i.e, an interval of the domain and the range. For convenience, we use the convention that the entire permutation is not itself an interval.

Intervals of permutations arise naturally through the process of \emph{inflation}: an inflation of a permutation $\pi$ is a permutation formed by replacing some of the points of $\pi$ by other permutations (with appropriate adjustments of values so that the result is a permutation). The permutation which results from inflating a permutation $\alpha$ of length $k$ by subpermutations $\tau_1, \ldots, \tau_k$ is denoted by $\alpha[\tau_1, \ldots, \tau_k]$. For example, $4371265  = 2413[21, 1, 12, 21]$, as shown in Figure~\ref{figure:perm-ex-2}.

The pattern containment relation is a partial order on the set of all permutations.  It admits eight automorphisms corresponding to the isometries of the square; for example, inversion of permutations corresponds to reflection over the line $y=x$ (see \cite{simion:restricted-permutations} for a more comprehensive discussion).  The partial order is studied through its lower ideals, i.e., sets of permutations which are closed downwards under pattern containment. These sets are called \emph{classes}.  Commonly, a permutation class $\C$ is described by specifying the (unique) set of minimal permutations that do not belong to $\C$.  This set is called the \emph{basis} and we write $\C=\Av(B)$ to signify that $\C$ is the set of permutations that \emph{avoid} (do not contain) any of the permutations of the basis $B$.

The first permutation classes to be studied are the so-called \emph{principal} classes --- those whose a basis consists of a single permutation. All of our results will be confined to this case.  Despite their simple definition, the structure of principal permutation classes is not very well understood.  In particular the vast majority of permutation classes that have been enumerated are not principal.

Recent successes in permutation class enumeration (for example: \cite{atkinson:(3+1)-avoiding, albert:enumeration-2143-4231, albert:enumeration-three-classes-grid, albert:inflations-case-studies, pantone:enumeration-3124-4312, bona:pattern-avoiding-involutions}) have relied heavily on the notion of a \emph{simple} permutation.  A permutation $\sigma$ is simple if it has no intervals other than those consisting of single points.  Their significance in the theory of permutation classes in due in large part to the following result.

\begin{proposition}[Albert and Atkinson~\cite{albert:simple-permutations}]
	Every permutation $\pi$ is the inflation of a unique simple permutation $\sigma$.  If $|\sigma|>2$ then the maximal intervals of $\pi$ are disjoint and $\pi$ is obtained from $\sigma$ by inflating its points with these maximal intervals.
\end{proposition}

When $\sigma = 12$ or $\sigma=21$ in this proposition, we say that $\pi$ is \emph{sum-decomposable} or \emph{skew-decomposable}, respectively. Sums and skew-sums are written using the notation $\alpha \oplus \beta = 12[\alpha, \beta]$ and $\alpha \ominus \beta = 21[\alpha, \beta].$ If $\pi$ is not sum-decomposable, it is said to be \emph{sum-indecomposable}, with \emph{skew-indecomposable} defined similarly. We use the term \emph{decomposable} to mean either sum-decomposable or skew-decomposable, and we use the term \emph{indecomposable} to mean neither sum-decomposable nor skew-decomposable.

Simple permutations are quite common: Albert, Atkinson, and Klazar~\cite{albert:enumeration-simple-perms} showed that within the set of all permutations, they have density $1/e^2$.  However, it is often the case that within any particular permutation class (apart from the set of all permutations) their density is far lower; no satisfactory explanation of this phenomenon is known.  Indeed there is no known permutation class whose simple permutations have positive density within the class as a whole. The unexpected low density of simple permutations in a permutation class $\C$ is frequently a consequence of the simple permutations lying in a proper subclass $\C'$ of $\C$. This property may enable the structure of $\C$ to be determined and thereby its enumeration: briefly, $\C'$ is an easier class to work with and the entirety of $\C$ can be recovered by inflation.

Rather than think about the simple permutations of a given class, one can think about classes which contain a given set $S$ of simple permutations; for ease of exposition, we assume that $S$ is closed under taking simple subpermutations. On one end of the spectrum is the downward closure of $S$, i.e., the set of all subpermutations of elements of $S$, denoted $\Cl(S)$, which of course is the smallest class containing $S$. On the other end we encounter the notion of the \emph{substitution closure} of a class: the substitution closure of $\C$, denoted $\langle \C \rangle$, is the largest class containing the same simple permutations. With this notation, the largest class containing exactly the simples in $S$ is $\langle \Cl(S) \rangle$.

As mentioned above, it is frequently easier to enumerate and describe a class if its simple permutations are actually contained in a smaller class. This leads us to a key definition. A permutation class $\C$ is said to be \emph{deflatable} if its simple permutations lie in a smaller class, i.e., if $\C \subseteq \langle \D \rangle$ for some proper subclass $\D$ of $\C$. This term is intended to convey that the proper subclass $\D$ can be obtained by, in a sense, reversing the operation of inflating simple permutations. Moreover, it follows that $\C$ is not deflatable if and only if $\C$ is equal to the downward closure of its simple permutations.

Examples of deflatable and non-deflatable classes are readily given.  The class $\Av(231)$ is deflatable since 12 and 21 are its only simple permutations.  On the other hand $\Av(321)$ is not deflatable, a result which is more or less folkloric, but appears (essentially) in a paper by Albert, Atkinson, Brignall, Ru\v{s}kuc, Smith, and West~\cite[Proposition 6]{albert:growth-rates-subclasses-321}.

Now we can pose our central question: 
\emph{for which permutations $\pi$ is $\Av(\pi)$ deflatable}?
We give partial answers only.  In Section~\ref{section:notdef}, we will consider classes $\Av(\pi)$ for which $\pi$ is decomposable -- we will show that many of these classes are not deflatable. Section~\ref{section:def} presents a test for deflatability and uses this test to provide several examples of deflatable permutation classes, including an infinite family. Lastly, Section~\ref{section:concl} proves non-deflatability in a special case and poses some questions.

\section{Preliminary Lemmas}
\label{section:prelim-lemmas}

In the following chapters, we will sometimes restrict our focus to indecomposable permutations. The next lemma shows that this does not lose us much generality.

\begin{lemma}
	\label{lemma:embed-indec}
	Every permutation in $\Av(\pi)$ can be embedded into an indecomposable permutation in $\Av(\pi)$ unless $\pi \in \{1, 12, 21, 132, 213, 231, 312\}$.
\end{lemma}
\begin{proof}
	Let $\omega \in \Av(\pi)$. We first handle the case where $\pi$ has a corner point, i.e., $\pi$ has one of the forms $1 \oplus \tau$, $\tau \oplus 1$, $1 \ominus \tau$, or $\tau \ominus 1$. Assume further that $\pi$ starts with $1$; the other three cases follow symmetrical arguments.
	
	We first embed $\omega$ into the sum-indecomposable permutation $\hat{\omega} = \omega \ominus 1$. By the assumption that $\pi$ has the form $1 \oplus \tau$, it is clear that $\hat{\omega} \in \Av(\pi)$. Let us be explicit, just once, about this sort of remark. Suppose that $1 \oplus \tau \leq \omega \ominus 1$. Then there is a subset of $\omega \ominus 1$ whose elements have pattern $1 \oplus \tau$. In particular, the leftmost element of this subset is its least element. Therefore, the last element of $\omega \ominus 1$ cannot be in the set, and in fact $1 \oplus \tau \leq \omega$. Consider the skew-decomposition $\hat{\omega} = \omega_1 \ominus \cdots \ominus \omega_k$ such that each $\omega_i$ is itself skew-indecomposable. Form a new skew-indecomposable permutation $\widebar{\omega} = \widebar{\omega_1} \ominus \cdots \ominus \widebar{\omega_k}$, where $\widebar{\omega_i} = 12$ if $\omega_i = 1$, and $\widebar{\omega_i} = \omega_i$ otherwise. Lastly, obtain an indecomposable permutation $\zeta$ containing $\omega$ by taking each pair $(\widebar{\omega_i}, \widebar{\omega_{i+1}})$ of skew components of $\widebar{\omega}$ and linking them together by inserting an entry just before the final point of $\widebar{\omega_i}$ and just below the topmost point of $\widebar{\omega_{i+1}}$. The only permutations beginning with $1$ that can be introduced by this step are $1$, $12$, and $132$. Therefore $\zeta \in \Av(\pi)$. Figure~\ref{figure:embed-indec} gives an example of performing these steps to $\omega = 564213$. This completes the proof in the case that $\pi$ has a corner point.
	\begin{figure}
	\begin{center}
		\begin{tikzpicture}[scale=0.3, baseline=(current bounding box.south)]
			\fill [color=lightgray] (0.5,4.5) rectangle (2.5, 6.5);
			\fill [color=lightgray] (2.5,3.5) rectangle (3.5, 4.5);
			\fill [color=lightgray] (3.5,0.5) rectangle (6.5, 3.5);
			\draw[step=1cm,gray,very thin] (0.0,0.0) grid (7,7);
			\draw (0,0) rectangle (7,7);
			\fill (1,5) circle (0.2cm);
			\fill (2,6) circle (0.2cm);
			\fill (3,4) circle (0.2cm);
			\fill (4,2) circle (0.2cm);
			\fill (5,1) circle (0.2cm);
			\fill (6,3) circle (0.2cm);
			\node at (3.5,-1) {$\omega = 564213$};
			\node at (8.5,4) {$\Longrightarrow$};
		\end{tikzpicture}\!
		\begin{tikzpicture}[scale=0.3, baseline=(current bounding box.south)]
			\fill [color=lightgray] (0.5,5.5) rectangle (2.5, 7.5);
			\fill [color=lightgray] (2.5,4.5) rectangle (3.5, 5.5);
			\fill [color=lightgray] (3.5,1.5) rectangle (6.5, 4.5);
			\fill [color=lightgray] (6.5,0.5) rectangle (7.5, 1.5);
			\draw[step=1cm,gray,very thin] (0.0,0.0) grid (8,8);
			\draw (0,0) rectangle (8,8);
			\fill (1,6) circle (0.2cm);
			\fill (2,7) circle (0.2cm);
			\fill (3,5) circle (0.2cm);
			\fill (4,3) circle (0.2cm);
			\fill (5,2) circle (0.2cm);
			\fill (6,4) circle (0.2cm);
			\fill (7,1) circle (0.2cm);
			\node at (4,-1) {$\widehat{\omega} = 6753241$};
			\node at (9.5,4) {$\Longrightarrow$};
		\end{tikzpicture}\!
		\begin{tikzpicture}[scale=0.3, baseline=(current bounding box.south)]
			\fill [color=lightgray] (0.5,7.5) rectangle (2.5, 9.5);
			\fill [color=lightgray] (2.5,5.5) rectangle (4.5, 7.5);
			\fill [color=lightgray] (4.5,2.5) rectangle (7.5, 5.5);
			\fill [color=lightgray] (7.5,0.5) rectangle (9.5, 2.5);
			\draw[step=1cm,gray,very thin] (0.0,0.0) grid (10,10);
			\draw (0,0) rectangle (10,10);
			\fill (1,8) circle (0.2cm);
			\fill (2,9) circle (0.2cm);
			\fill (3,6) circle (0.2cm);
			\fill (4,7) circle (0.2cm);
			\fill (5,4) circle (0.2cm);
			\fill (6,3) circle (0.2cm);
			\fill (7,5) circle (0.2cm);
			\fill (8,1) circle (0.2cm);
			\fill (9,2) circle (0.2cm);
			\node at (5,-1) {$\widebar{\omega} = 896743512$};
			\node at (11.5,4) {$\Longrightarrow$};
		\end{tikzpicture}\hspace*{-.16in}
		\begin{tikzpicture}[scale=0.3, baseline=(current bounding box.south)]
			\fill [color=lightgray] (0.5,10.5) rectangle (3.5, 12.5);
			\fill [color=lightgray] (3.5,7.5) rectangle (6.5, 10.5);
			\fill [color=lightgray] (6.5,3.5) rectangle (10.5, 7.5);
			\fill [color=lightgray] (10.5,0.5) rectangle (12.5, 3.5);
			\draw[step=1cm,gray,very thin] (0.0,0.0) grid (13,13);
			\draw (0,0) rectangle (13,13);
			\fill (1,11) circle (0.2cm);
			\fill (2,9) circle (0.2cm);
			\fill (3,12) circle (0.2cm);
			\fill (4,8) circle (0.2cm);
			\fill (5,6) circle (0.2cm);
			\fill (6,10) circle (0.2cm);
			\fill (7,5) circle (0.2cm);
			\fill (8,4) circle (0.2cm);
			\fill (9,2) circle (0.2cm);
			\fill (10,7) circle (0.2cm);
			\fill (11,1) circle (0.2cm);
			\fill (12,3) circle (0.2cm);
			\node at (6.5,-1) {$\zeta = 11\;9\;12\;8\;6\;10\;5\;4\;2\;7\;1\;3$};
		\end{tikzpicture}
	\end{center}
	\caption{The progression from $\omega$ to $\zeta$ as described in the proof of Lemma~\ref{lemma:embed-indec}.}
	\label{figure:embed-indec}
\end{figure}
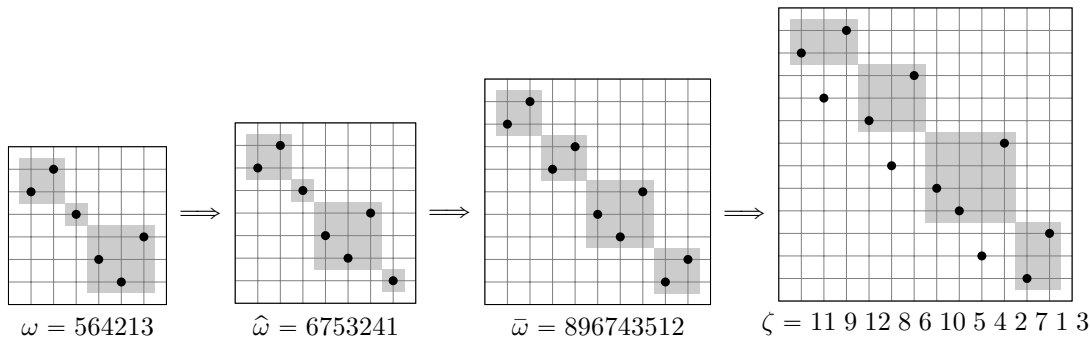

Assume now that $\pi$ has no corner points. It follows that the outer points (the top- bottom- left- and right-most entries) of $\pi$ form one of the patterns $2143$, $2413$, $3142$, or $3412$. By appealing to a symmetry if necessary, we can assume that these outer points do not form a $2413$ pattern. Form an indecomposable permutation $\zeta$ by adding outer points to $\pi$ to form a $2413$, as shown in Figure~\ref{figure:embed-indec-2413}. 
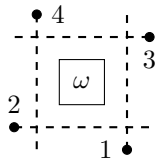
\begin{figure}
	\begin{center}
		\begin{tikzpicture}[scale=0.3, baseline=(current bounding box).center]
			\draw (2,2) rectangle (4,4) node[midway] {$\omega$};
			\draw[fill] (0,1) circle (.2) node[above=2pt] {$2$};
			\draw[fill] (5,0) circle (.2) node[left=2pt] {$1$};
			\draw[fill] (1,6) circle (.2) node[right=2pt] {$4$};
			\draw[fill] (6,5) circle (.2) node[below=2pt] {$3$};
			\draw[thick, dashed] (0,1) -- (6,1);
			\draw[thick, dashed] (5,0) -- (5,6);
			\draw[thick, dashed] (6,5) -- (0,5);
			\draw[thick, dashed] (1,6) -- (1,0);
		\end{tikzpicture}
		
	\end{center}
	\caption{The embedding of $\omega$ into $\zeta$ in the case where $\omega$ does not have a corner point as described in the proof of Lemma~\ref{lemma:embed-indec}.}
	\label{figure:embed-indec-2413}
\end{figure}
For convenience, we refer to these points as the $2$, the $4$, the $1$, and the $3$. Suppose toward a contradiction that $\zeta$ contains an occurrence of $\pi$. Since $\omega \in \Av(\pi)$, at least one of the new outer points must be involved in the occurrence of $\pi$. We will assume that the $2$ is involved, and all other cases follow symmetrically.  However, since we assumed that $\pi$ does not have a corner point, the outer point $1$ must also be involved. Similarly, $1$ would be a corner point unless the outer point $3$ were involved, and again $3$ would be a corner point unless the outer point $4$ were involved. Therefore, the outer points of $\pi$ form a $2413$, a contradiction.
\end{proof}

Suppose we wish to prove that a class $\Av(\pi)$ is not deflatable. To accomplish this, we have to show that each permutation $\omega \in \Av(\pi)$ can be embedded into some simple permutation of $\Av(\pi)$. By Lemma~\ref{lemma:embed-indec}, we may assume that $\omega$ is indecomposable; that is, it is an inflation of a simple permutation of length strictly greater than $2$. We henceforth always assume that $\omega$ is indecomposable unless otherwise stated. This allows us to speak of maximal intervals without fear that they may overlap.

If it happens that all maximal intervals of $\omega$ have length $1$, then $\omega$ is already a simple permutation (recall that we do not allow the entire permutation to be, itself, an interval). Thus, we can restrict ourselves to $\omega$ containing a longest interval $\alpha$ of size at least 2 (choosing one arbitrarily if there are several longest intervals). Then, we shall embed $\omega$ into a one-point extension $\omega^+$ by adding a new point $x$ which cuts the interval $\alpha$. Our aim will be to replace $\omega$ with $\omega^+$, ensuring that it too is indecomposable, lies in $\Av(\pi)$, and is closer to being a simple permutation. A permutation $\omega$ for which this is possible will be called \emph{breakable}. Replacing $\omega^+$ with $\omega$ and repeating this step, we shall eventually embed $\omega$ in a simple permutation of $\Av(\pi)$. 

Even if this basic step can be repeated whenever $\omega$ is not simple, we need a precise measure of being closer to a simple permutation. The measure we use is 
	\[\SD(\omega)  = \ds\sum_\gamma \left(|\gamma|-1\right),\]
where the summation is over all maximal intervals $\gamma$ of $\omega$, and $|\gamma|$ denotes the number of elements in $\gamma$. Our constructions will assure that $\SD(\omega)$ decreases. Clearly then the basic step can only be carried out finitely many times before $\SD(\omega)$ becomes equal to $0$, and hence the permutation becomes simple.

As we shall now see, there is a general criterion that guarantees that the basic step results in another indecomposable permutation and reduces the measure $\SD(\omega)$.

\begin{lemma}
	\label{lemma:MI}
	Suppose that $\omega$ is an indecomposable permutation with an interval $\alpha$ of maximum length $\ell > 1$. Suppose that $\omega^+$ is an extension of $\omega$ by a point $x$ that cuts $\alpha$ and that $\alpha \cup \{x\}$ does not form an interval in $\omega^+$. Then, $\omega^+$ is indecomposable and $\SD(\omega^+) < \SD(\omega)$.
\end{lemma}
\begin{proof}
	It is evident that $\omega^+$ is indecomposable because $x$ cuts $\alpha$ and so the new point $x$ is not a corner point.
	
	We compare the maximal intervals of $\omega$ with the maximal intervals of $\omega^+$. Without loss of generality, we may assume that $x$ cuts $\alpha$ by position so, since  $\alpha \cup \{x\}$ does not form an interval in $\omega^+$, $\alpha$ is separated from $x$ by value. Then, $x$ cuts no other maximal interval of $\omega$ by position, and cuts at most one maximal interval of $\omega$ by value.
	
	The maximal interval $\beta$ of $\omega^+$ that contains $x$ is just the singleton $\{x\}$, for if there were another point in this interval, it would have at least one positional neighbor $u$ in $\alpha$. Take $v$ to be a point separating $x$ from $\alpha$ by value. Then $v$ must lie in $\beta$ since $\beta$ contains points of value less than $v$ and greater than $v$. Now, $\beta \setminus \{x\}$ would be an interval of $\omega$ that contains points from two distinct maximal intervals, a contradiction. Clearly, the maximal interval $\{x\}$ contributes $0$ to $\SD (\omega^+)$.
	
	Now consider a maximal interval $\beta$ of $\omega^+$ that does not contain $x$. Then, $\beta$ is also an interval of $\omega$ and so $\beta$ is contained in a maximal interval $\gamma$ of $\omega$. If $x$ does not cut $\gamma$, then $\gamma$ is also an interval of $\omega^+$ and hence, by the maximality of $\beta$, we have $\beta = \gamma$. Thus, such intervals contribute equal amounts to both $\SD(\omega)$ and $\SD(\omega^+)$. However, if $x$ cuts $\gamma$ (which certainly happens if $\gamma = \alpha$), then $\gamma$ is not an interval of $\omega^+$ and so $\gamma$ will be a proper union $\beta_1 \cup \cdots \cup \beta_k$ of more than one maximal interval. Since the union is proper,
		\[
			\sum_{i=1}^k \left(|\beta_i|-1\right) = |\gamma|-k < |\gamma|-1.
		\]
It follows that $\SD(\omega^+) < \SD(\omega)$, as desired.
\end{proof}

Lemma~\ref{lemma:embed-indec} and Lemma~\ref{lemma:MI} set the stage for all of our proofs that a principal class $\Av(\pi)$ is not deflatable. They show that non-deflatability will follow if, for every indecomposable permutation $\omega \in \Av(\pi)$ with non-trivial maximal interval $\alpha$, we can find a one-point extension by a cut point $x$ of $\alpha$ to a permutation $\omega^+ \in \Av(\pi)$ such that $\alpha \cup \{x\}$ does not form an interval, i.e., if every indecomposable non-simple permutation of $\Av(\pi)$ is breakable. We shall call classes with this property \emph{extendible}. For convenience, we record this below.

\begin{lemma}
	\label{lemma:extendible}
	If a class $\C$ is extendible, then it is not deflatable.
\end{lemma}

\section{Non-Deflatable Permutation Classes}
\label{section:notdef}

This section focuses mainly on decomposable permutations $\pi$. Because deflatability is invariant over the symmetries of a permutation, we choose to consider only sum-decomposable permutations. Each of the proofs in this section proceeds with the following setup.

\begin{figure}
	\minipage{0.5\textwidth}
		\begin{center}
			\begin{tikzpicture}[scale = .6]
				\draw[draw=none, fill=black!15] (2,0) rectangle ++(2,2);
				\draw[draw=none, fill=black!15] (0,2) rectangle ++(2,2);
				\draw[draw=none, fill=black!15] (2,4) rectangle ++(2,2);
				\draw[draw=none, fill=black!15] (4,2) rectangle ++(2,2);
				\foreach \i in {2,4} {
					\draw (\i,0) -- (\i,6);
					\draw (0,\i) -- (6,\i);
				}
				\node at (1,1) {$\beta$};
				\node at (5,1) {$\epsilon$};
				\node at (5,5) {$\delta$};
				\node at (1,5) {$\gamma$};
				\node at (3,3) {$\alpha$};
			\end{tikzpicture}
		\end{center}
		\caption{Diagram of an indecomposable $\omega$ with a longest maximal interval $\alpha$. We label the four corner regions as $\beta$, $\gamma$, $\delta$, and $\epsilon$.}
		\label{figure:typical-omega}
	\endminipage\hfill
	\minipage{0.5\textwidth}
		\begin{center}
			\begin{tikzpicture}[scale = .6]
				\draw[draw=none, fill=black!15] (2,0) rectangle ++(2,2);
				\draw[draw=none, fill=black!15] (0,2) rectangle ++(2,2);
				\draw[draw=none, fill=black!15] (2,4) rectangle ++(2,2);
				\draw[draw=none, fill=black!15] (4,2) rectangle ++(2,2);
				\foreach \i in {2,4} {
					\draw (\i,0) -- (\i,6);
					\draw (0,\i) -- (6,\i);
				}
				\draw[fill] (3,2) circle (.08) node[above right] {$a$};
				\draw[fill] (1.5,1) circle (.08) node[right] {$b$};
				\draw[fill] (1.3,2.2) circle (.08) node[above left] {$x$};
				\draw[dashed] (1.3,1)--(1.3,2.2)--(3,2.2);
				\node at (3,3) {$\alpha$};
			\end{tikzpicture}
		\end{center}
		\caption{The permutation $\omega^+$ formed by inserting the entry $x$ into $\omega$ just to the left of $b$ and just above $a$.\\}
		\label{figure:beta-lambda-1}
	\endminipage
\end{figure}

Let $\pi$ be sum-decomposable and suppose $\omega \in \Av(\pi)$ is indecomposable and non-simple. Let $\alpha$ be a longest maximal interval of $\omega$. Then, $\omega$ can be depicted as in Figure~\ref{figure:typical-omega}, where the shaded regions signify that no entries cut $\alpha$ by either position or value. We will always refer to the regions $\alpha$, $\beta$, $\gamma$, $\delta$, and $\epsilon$ as shown in Figure~\ref{figure:typical-omega}.
Suppose that $\pi= \lambda \oplus \mu \oplus \rho$ (where we allow $\mu$ to be possibly empty).

We first show that if $\lambda \leq \beta$ or $\rho \leq \delta$, then $\omega$ has a one-point extension splitting $\alpha$, so that in all future proofs we can assume $\beta \in \Av(\lambda)$ and $\delta \in \Av(\rho)$.

Suppose that $\lambda \leq \beta$. Let $b$ be the rightmost point of the leftmost occurrence of $\lambda$ in $\beta$. Let $a$ be the bottommost point of $\alpha$. Insert a new entry $x$ just to the left of $b$ and just above $a$ to form $\omega^+$, as in Figure~\ref{figure:beta-lambda-1}. It is clear that $\alpha \cup \{x\}$ is not an interval because they are separated by $b$. Furthermore, suppose that the insertion of $x$ introduced an occurrence of $\pi$ in $\omega^+$. Then, $x$ itself must be involved, otherwise $\omega$ would have contained an occurrence of $\pi$. However, $x$ cannot play a role in the $\lambda$ part of $\pi$, for otherwise there would be an occurrence of $\mu \oplus \rho$ above and to the right of $x$, and hence above and to the right of the occurrence of $\lambda$ which ends with $b$. This would imply that $\omega$ contained an occurrence of $\pi$, a contradiction. Moreover, if $x$ played the role in the $\mu \oplus \rho$ part of $\pi$ in $\omega^+$, then there would be an occurrence of $\lambda$ below and to the left of $x$, contradicting our choice of $b$. Thus, $\omega^+ \in \Av(\pi)$.

Therefore, if $\lambda \leq \beta$, it follows that $\omega$ is breakable. A symmetric argument shows that $\omega$ is breakable if $\rho \leq \delta$. Thus, for $\pi = \lambda \oplus \mu \oplus \rho$ (with $\mu$ possibly empty), when trying to show that $\omega \in \Av(\pi)$ is breakable, we may always assume that $\beta \in \Av(\lambda)$ and $\delta \in \Av(\rho)$. Additionally, the indecomposability of $\omega$ implies that $\gamma$ and $\epsilon$ are not both empty. 
We can now begin to investigate which decomposable permutations $\pi$ lead to deflatable classes $\Av(\pi)$ and which do not. Most of the remainder of this section is dedicated to showing that ``most'' such classes are non-deflatable.

\begin{theorem}
	\label{theorem:three-comp-sum}
	Let $\pi = \lambda \oplus \mu \oplus \rho$ with all three summands non-empty. Then, $\Av(\pi)$ is not deflatable.
\end{theorem}
\begin{proof}
	Let $\omega \in \Av(\pi)$ be indecomposable. As above, we can assume that $\beta \in \Av(\lambda)$, $\delta \in \Av(\rho)$. At least one of $\gamma$ and $\epsilon$ is non-empty. Since $\pi$ satisfies the conditions of the hypothesis if and only if $\pi^{-1}$ does, we can assume that $\gamma$ is non-empty without loss of generality. Let $c$ be the rightmost entry of $\gamma$ and let $a$ be the topmost entry of $\alpha$. Form $\omega^+$ by inserting an entry $x$ into $\omega$ that lies just to the left of $c$ and just below $a$, as in Figure~\ref{figure:three-comp-sum-1}.

\begin{figure}
	\minipage{0.5\textwidth}
		\begin{center}
			\begin{tikzpicture}[scale = .6]
				\draw[draw=none, fill=black!15] (2,0) rectangle ++(2,2);
				\draw[draw=none, fill=black!15] (0,2) rectangle ++(2,2);
				\draw[draw=none, fill=black!15] (2,4) rectangle ++(2,2);
				\draw[draw=none, fill=black!15] (4,2) rectangle ++(2,2);
				\draw[draw=none, fill=black!15] (1.5,4) rectangle ++(.5,2);
				\foreach \i in {2,4} {
					\draw (\i,0) -- (\i,6);
					\draw (0,\i) -- (6,\i);
				}
				\draw[fill] (3,4) circle (.08) node[above] {$a$};
				\draw[fill] (1.5,5) circle (.08) node[above left=-1pt] {$c$};
				\draw[fill] (1.3,3.8) circle (.08) node[below left] {$x$};
				\draw[dashed] (1.3, 5)--(1.3,3.8)--(2.8,3.8);
				\node at (3,3) {$\alpha$};
			\end{tikzpicture}
		\end{center}
		\caption{The permutation $\omega^+$ formed by inserting the entry $x$ into $\omega$ just to the left of $c$ and just below $a$, in the proof of Theorem~\ref{theorem:three-comp-sum}}
		\label{figure:three-comp-sum-1}
	\endminipage\hfill
	\minipage{0.5\textwidth}
		\begin{center}
			\begin{tikzpicture}[scale = .6]
				\draw[draw=none, fill=black!15] (2,0) rectangle ++(2,2);
				\draw[draw=none, fill=black!15] (0,2) rectangle ++(2,2);
				\draw[draw=none, fill=black!15] (2,4) rectangle ++(2,2);
				\draw[draw=none, fill=black!15] (4,2) rectangle ++(2,2);
				\draw[draw=none, fill=black!15] (.8,0) rectangle ++(1.2,2);
				\draw[draw=none, fill=black!15] (0,4) rectangle ++(2,.6);
				\foreach \i in {2,4} {
					\draw (\i,0) -- (\i,6);
					\draw (0,\i) -- (6,\i);
				}
				\node at (3,3) {$\alpha$};
				\draw[fill=black] (1.2, 4.6) circle (.08) node[below] {$c$};
				\draw[fill=black] (2, 3) circle (.08) node[left] {$a$};
				\draw[fill=black] (2.2, 4.8) circle (.08) node[right] {$x$};
				\draw[dashed] (1.2, 4.8) -- (2.2, 4.8) -- (2.2, 3);
			\end{tikzpicture}
		\end{center}
		\caption{The diagram for Case 1 in the proof of Theorem~\ref{theorem:33proof}.\ \\\ \\}
		\label{figure:33proof-1}
	\endminipage
\end{figure}

Suppose that $x$ is part of an occurrence of $\pi$ in $\omega^+$. If $x$ plays a role in the $\rho$ part of $\pi$, then $\lambda \leq \beta$, a contradiction. If $x$ plays a role in the $\lambda$ part of $\pi$, then the $\mu \oplus \rho$ part of the occurrence of $\pi$ lies among $\{a,c\} \cup \delta$. At least one of $a$ or $c$ must belong to the $\rho$ part of this occurrence, since otherwise we would have $\rho \leq \delta$. However, $a \in \rho$ and $c \in \rho$ are each impossible since $\mu$ is non-empty and $a$ and $c$ are, respectively, the first and lowest elements of $\{a,c\} \cup \delta$. Hence, $x$ is not part of an occurrence of $\pi$, which shows that $\omega$ is breakable. Since $\omega$ was an arbitrary indecomposable permutation, $\Av(\pi)$ is extendible, and so, by Lemma~\ref{lemma:extendible}, is not deflatable.
\end{proof}

As the above theorem did not require $\lambda$, $\mu$, and $\rho$ to be sum-indecomposable, it handles all sum-decomposable permutations except for those of the form $\pi = \lambda \oplus \rho$ with $\lambda$ and $\rho$ sum-indecomposable. The next theorem begins to handle this case.

\begin{theorem}
	\label{theorem:33proof}
	Let $\pi = \lambda \oplus \rho$, with $|\lambda|, |\rho| \geq 2$. Then, $\Av(\pi)$ is not deflatable. 
\end{theorem}
\begin{proof}
	Theorem~\ref{theorem:three-comp-sum} allows us to assume that $\lambda$ and $\rho$ are sum-indecomposable.	Suppose that $\omega \in \Av(\pi)$ is indecomposable. We wish to show that $\omega$ is breakable. To this end, choose a largest maximal interval $\alpha$ of $\omega$ and let $\beta$, $\gamma$, $\delta$ and $\epsilon$ be as in Figure \ref{figure:typical-omega}.  By our previous arguments we may also suppose that $\beta \in \Av(\lambda)$ and $\delta \in \Av(\rho)$. We can assume by symmetry that $\gamma$ is non-empty (for the same reason as in the previous theorem). Furthermore, at least one of $\beta$ and $\delta$ is non-empty, as otherwise $\omega$ is skew-decomposable. We consider a division into cases.
		
	 \bigskip
 
	\emph{Case 1: $\beta$ is empty, or the last entry of $\beta$ precedes the smallest entry of $\gamma$}

Let $c$ be the smallest entry of $\gamma$. Suppose that all entries of $\beta$ lie to the left of $c$ (it is permissible that $\beta$ be empty). Let $a$ be the leftmost entry in $\alpha$. We will show that $\alpha$ can be split by an entry $x$ placed just above $c$ and just to the right of $a$, as in Figure~\ref{figure:33proof-1}.

Suppose that the placement of $x$ introduces an occurrence of $\pi$. If $x$ were an entry of the $\lambda$ part of $\pi$, then the $\rho$ part of $\pi$ would lie entirely above it and to its right. This would force $\rho$ to be entirely contained in $\delta$, a contradiction to an earlier assumption. Therefore, $x$ must be an entry in the $\rho$ part of $\pi$. It follows that the $\lambda$ part of $\pi$ occurs in $\{a,c\} \cup \beta$.

If the occurrence of $\lambda$ contains the point $c$, then the occurrence of $\rho$ contains $x$ as its first and least entry, a contradiction to the assumption that $\rho$ is sum-indecomposable. If the occurrence of $\lambda$ contains the point $a$ (but not the point $c$), then $a$ is the last and greatest entry of $\lambda$, contradicting that $\lambda$ is sum-indecomposable. Hence the occurrence of $\lambda$ is contained entirely within $\beta$, contradicting a previous assumption.


\bigskip

\emph{Case 2: The smallest entry of $\gamma$ precedes the last entry of $\beta$}

Let $b$ be the rightmost entry of $\beta$ and let $c'$ be the rightmost entry of $\gamma$. We handle two cases: either $c'$ precedes $b$ or $b$ precedes $c'$.

\smallskip

\emph{Case 2a: The last entry of $\gamma$ precedes the last entry of $\beta$}

Consider a splitting entry $x$ which lies just to the right of $c'$ and just below $a'$, as in Figure~\ref{figure:33proof-2}. If the insertion of $x$ creates an occurrence of $\pi$, then $x$ lies in the $\lambda$ part of $\pi$ -- otherwise the $\lambda$ part of $\pi$ lies entirely in $\beta$, a contradiction. Hence, the occurrence of the $\rho$ part is contained in $\{a'\} \cup \delta$. Since $\delta \in \Av(\rho)$, the occurrence of $\rho$ must contain the point $a'$, implying that $\rho$ is sum-decomposable. This contradicts our previous assumption. Therefore, $x$ splits $\alpha$ without introducing an occurrence of $\pi$. Note that this case did not require that $c \neq c'$ nor $a \neq a'$. 
\begin{figure}
	\minipage{0.5\textwidth}
		\begin{center}
			\begin{tikzpicture}[scale = .6]
				\draw[draw=none, fill=black!15] (2,0) rectangle ++(2,2);
				\draw[draw=none, fill=black!15] (0,2) rectangle ++(2,2);
				\draw[draw=none, fill=black!15] (2,4) rectangle ++(2,2);
				\draw[draw=none, fill=black!15] (4,2) rectangle ++(2,2);
				\draw[draw=none, fill=black!15] (1.5,0) rectangle ++(.5,2);
				\draw[draw=none, fill=black!15] (0,4) rectangle ++(2,.6);
				\draw[draw=none, fill=black!15] (1.2,4) rectangle ++(.8,2);
				\foreach \i in {2,4} {
					\draw (\i,0) -- (\i,6);
					\draw (0,\i) -- (6,\i);
				}
				\node at (3,3) {$\alpha$};
				\draw[fill=black] (.8, 4.6) circle (.08) node[below] {$c$};
				\draw[fill=black] (1.2, 5.2) circle (.08) node[above=2pt,left] {$c'$};
				\draw[fill=black] (2, 3) circle (.08) node[left] {$a$};
				\draw[fill=black] (1.4, 3.8) circle (.08) node[below] {$x$};
				\draw[fill=black] (1.5, 1) circle (.08) node[left] {$b$};
				\draw[fill=black] (3, 4) circle (.08) node[right=2pt, above] {$a'$};
				\draw[dashed] (3, 3.8) -- (1.4, 3.8) -- (1.4, 5.2);
			\end{tikzpicture}

		\end{center}
		\caption{The diagram for Case 2a in the proof of Theorem~\ref{theorem:33proof}.}
		\label{figure:33proof-2}
	\endminipage\hfill
	\minipage{0.5\textwidth}
		\begin{center}
			\begin{tikzpicture}[scale = .6]
				\draw[draw=none, fill=black!15] (2,0) rectangle ++(2,2);
				\draw[draw=none, fill=black!15] (0,2) rectangle ++(2,2);
				\draw[draw=none, fill=black!15] (2,4) rectangle ++(2,2);
				\draw[draw=none, fill=black!15] (4,2) rectangle ++(2,2);
				\draw[draw=none, fill=black!15] (1.2,0) rectangle ++(.8,2);
				\draw[draw=none, fill=black!15] (0,4) rectangle ++(2,.6);
				\draw[draw=none, fill=black!15] (1.5,4) rectangle ++(.5,2);
				\foreach \i in {2,4} {
					\draw (\i,0) -- (\i,6);
					\draw (0,\i) -- (6,\i);
				}
				\node at (3,3) {$\alpha$};
				\draw[fill=black] (.8, 4.6) circle (.08) node[below] {$c$};
				\draw[fill=black] (1.5, 5.2) circle (.08) node[above=2pt,left] {$c'$};
				\draw[fill=black] (2, 3) circle (.08) node[left] {$a$};
				\draw[fill=black] (1.2, 1) circle (.08) node[left] {$b$};
				\draw[fill=black] (3, 4) circle (.08) node[right=2pt, above] {$a'$};
				\draw[fill=black] (2.2, 4.8) circle (.08) node[right] {$x$};
				\draw[dashed] (.8, 4.8) -- (2.2, 4.8) -- (2.2, 3);
			\end{tikzpicture}
			\end{center}
		\caption{The first diagram for Case 2b in the proof of Theorem~\ref{theorem:33proof}.}
		\label{figure:33proof-3}
	\endminipage
\end{figure}

\smallskip

\emph{Case 2b: The last entry of $\beta$ precedes the last entry of $\gamma$}

Assume now that $b$ precedes $c'$. Let $x$ be a splitting entry which lies just to the right of $a$ and just above $c$.  (See Figure~\ref{figure:33proof-3}.)

Suppose that the insertion of the entry $x$ creates an occurrence of $\pi$. Since $x$ cannot lie in the $\lambda$ part of $\pi$ (as $\delta$ avoids $\rho$), $x$ must lie in the $\rho$ part of $\pi$. Hence, the $\lambda$ part of $\pi$ is contained in $\{a,c\} \cup \beta$. The point $c$ must be part of the occurrence of $\lambda$, since otherwise $\lambda$ is sum-decomposable. In fact, the point $a$ cannot be part of the $\lambda$ occurrence because this, together with $c$ also lying in the $\lambda$ part would force $x$ to be both the first and smallest entry of $\rho$, once again implying that $\rho$ is sum-decomposable. Therefore, there is an occurrence of $\lambda$ within $\{c\}\cup \beta$ that contains $c$. 

Now consider an alternative splitting point $x'$ placed just to the left of $c'$ and just below $a'$. Let $d$ be the lowest point of $\delta$. If $d$ is lower than $c$, then we can proceed by an argument symmetrical to Case 2a, and if $d > c'$, then we can proceed by an argument symmetrical to Case 1. Therefore, we may assume that $c$ is lower than $d$ and that $d$ is lower than $c'$. Figure~\ref{figure:33proof-4} shows the new splitting point $x'$, along with the occurrence of $\lambda$ which caused $x$ to fail as a splitting point. Assume also that $x'$ creates an occurrence of $\pi$. An argument symmetric to that of the previous paragraph by a reflection over the antidiagonal shows that there must be an occurrence of $\rho$ involving $c'$ and some points of $\delta$, as shown in Figure~\ref{figure:33proof-5}.

\begin{figure}
	\minipage{0.5\textwidth}
		\begin{center}
			\begin{tikzpicture}[scale = .6]
				\draw[draw=none, fill=black!15] (2,0) rectangle ++(2,2);
				\draw[draw=none, fill=black!15] (0,2) rectangle ++(2,2);
				\draw[draw=none, fill=black!15] (2,4) rectangle ++(2,2);
				\draw[draw=none, fill=black!15] (4,2) rectangle ++(2,2);
				\draw[draw=none, fill=black!15] (1.1,0) rectangle ++(.9,2);
				\draw[draw=none, fill=black!15] (0,4) rectangle ++(2,.6);
				\draw[draw=none, fill=black!15] (1.6,4) rectangle ++(.4,2);
				\draw[draw=none, fill=black!15] (4,4) rectangle ++(2,.8);
				\foreach \i in {2,4} {
					\draw (\i,0) -- (\i,6);
					\draw (0,\i) -- (6,\i);
				}
				\node at (3,3) {$\alpha$};
				\draw[fill=black] (.8, 4.6) circle (.08) node[below] {$c$};
				\draw[fill=black] (1.6, 5.2) circle (.08) node[above=2pt,left] {$c'$};
				\draw[fill=black] (2, 3) circle (.08) node[right] {$a$};
				\draw[fill=black] (1.5, 3.8) circle (.08) node[left=5pt, below right] {$x'$};
				\draw[fill=black] (1.1, 1) circle (.08) node[left] {$b$};
				\draw[fill=black] (5,4.8) circle (.08) node[above] {$d$};
				\draw[fill=black] (3, 4) circle (.08) node[right=2pt, above] {$a'$};
				\draw[dashed] (1.5, 3.8) -- (1.5, 5.2);
				\draw[dashed] (1.5, 3.8) -- (3, 3.8);
				\draw[thick] (-.2,-.2) rectangle (1.3,4.8);
				\node at (.2,.2) {$\lambda$};
			\end{tikzpicture}
		\end{center}
		\caption{A diagram showing the existence of an occurrence of $\lambda$ in $\omega$ as in Case 2b of Theorem~\ref{theorem:33proof}.}
		\label{figure:33proof-4}
	\endminipage\hfill
	\minipage{0.5\textwidth}
		\begin{center}
			\begin{tikzpicture}[scale = .6]
				\draw[draw=none, fill=black!15] (2,0) rectangle ++(2,2);
				\draw[draw=none, fill=black!15] (0,2) rectangle ++(2,2);
				\draw[draw=none, fill=black!15] (2,4) rectangle ++(2,2);
				\draw[draw=none, fill=black!15] (4,2) rectangle ++(2,2);
				\draw[draw=none, fill=black!15] (1.1,0) rectangle ++(.9,2);
				\draw[draw=none, fill=black!15] (0,4) rectangle ++(2,.6);
				\draw[draw=none, fill=black!15] (1.5,4) rectangle ++(.5,2);
				\draw[draw=none, fill=black!15] (4,4) rectangle ++(2,1);
				\foreach \i in {2,4} {
					\draw (\i,0) -- (\i,6);
					\draw (0,\i) -- (6,\i);
				}
				\node at (3,3) {$\alpha$};
				\draw[fill=black] (.8, 4.6) circle (.08) node[below] {$c$};
				\draw[fill=black] (1.5, 5.2) circle (.08) node[left=3pt, above right] {$c'$};
				\draw[fill=black] (2, 3) circle (.08) node[right] {$a$};
				\draw[fill=black] (1.1, 1) circle (.08) node[left] {$b$};
				\draw[fill=black] (5,5) circle (.08) node[above] {$d$};
				\draw[fill=black] (3, 4) circle (.08) node[right=2pt, above] {$a'$};
				\draw[thick] (-.2,-.2) rectangle (1.3,4.8);
				\draw[thick] (1.3, 4.8) rectangle (6.2, 6.2);
				\node at (.2,.2) {$\lambda$};
				\node at (5.8, 5.8) {$\rho$};
			\end{tikzpicture}
		\end{center}
		\caption{A diagram showing the existence of an occurrence of $\pi$ in $\omega$ as in Case 2b of Theorem~\ref{theorem:33proof}.}
		\label{figure:33proof-5}
	\endminipage
\end{figure}

Thus $\pi = \lambda \oplus \rho$ is contained in $\omega$, a contradiction.
\end{proof}

The two previous theorems show that for all decomposable permutations $\pi$ that are either the sum of three or more components, or are two-component sums with components both of size at least two, the class $\Av(\pi)$ is not deflatable. To complete the decomposable case, it remains, by symmetry, to handle the case where $\pi = 1 \oplus \rho$ for sum-indecomposable $\rho$. Initial intuition suggests that perhaps this is the easy case -- not only is that intuition false, it actually is not true that classes $\Av(\pi)$ with $\pi = 1 \oplus \rho$ are all non-deflatable, as we shall see in Section~\ref{section:def}. 

For the remainder of this section we concern ourselves solely with the case $\pi = 1 \oplus \rho$. We will at times need to refer to specific elements of this permutation and may do so either by position, e.g., ``the leftmost element of $\rho$'' or by value -- here keep in mind that for instance ``$2$'' would refer to the least element of $\rho$. Also, we assume $|\pi| = n$, so $n$ is always the maximum value (and of course occurs somewhere in $\rho$).

A set of two consecutive entries in a permutation is called a \emph{bond} if they are also consecutive in value (equivalently, a bond is a two element interval). If the elements of a bond form a $12$ pattern it is an \emph{increasing bond}, while a $21$ pattern is referred to as a \emph{decreasing bond}. For example, in the permutation $134652$, the entries $34$ form an increasing bond, while the entries $65$ form a decreasing bond.

The presence of bonds in $\rho$ seems to play an important role in the deflatability of $\pi = 1 \oplus \rho$. In particular, if $\rho$ lacks either an increasing bond or a decreasing bond, then $\Av(\pi)$ is not deflatable. We prove this in two parts, depending on whether $\rho$ starts with an ascent or starts with a descent.

\begin{theorem}
	\label{theorem:bonds-inc}
	Let $\pi = 1 \oplus \rho$, where $\rho$ is sum-indecomposable and starts with an ascent. If $\rho$ lacks either an increasing bond or a decreasing bond, then $\Av(\pi)$ is not deflatable.
\end{theorem}
\begin{proof}
	This proof is split into two separate cases in which $\rho$ either has no increasing bond or has no decreasing bond. As always, we assume $\omega \in \Av(\pi)$ is indecomposable, and that $\alpha$ is a largest maximal interval of $\omega$ with $\beta$, $\gamma$, $\delta$ and $\epsilon$ as in Figure \ref{figure:typical-omega}. Since $\lambda = 1$, we can assume that $\beta$ is empty in addition to assuming that $\delta \in \Av(\rho)$. Further $\delta$ is non-empty else $\omega$ would be skew-decomposable. 
	
	\bigskip
	
	\emph{Case 1: $\rho$ has no decreasing bond}
	
	Let $a$ be the leftmost entry of $\alpha$ and let $d$ be the leftmost entry of $\delta$. Form a one-point extension $\omega^+$ of $\omega$ by inserting a point $x$ just to the right of $a$ and just above $d$ as in Figure~\ref{figure:1plus-1}. If $\omega^+$ contains an occurrence of $\pi$, then it is immediately clear that $x$ must play a role in the $\rho$ part; otherwise $x$ plays the role of $1$ which would force $\rho \leq \delta$. Therefore, the $1$ of $\pi$ either lies in $\gamma$ or is equal to $a$.
	
	If the $1$ of $\pi$ lies in $\gamma$, then no entry in $\alpha$ can play a role in $\gamma$. Hence, $d$ also cannot play a role in $\rho$, as then it would be the entry immediately following $x$ and $\{x,d\}$ would form a decreasing bond. From this it follows that $\omega$ itself had an occurrence of $\pi$ in which $d$ played the same role as $x$ did in $\omega^+$, a contradiction. When this type of argument is used subsequently, we say that ``$d$ substitutes for $x$''.
	
	If the $1$ of $\pi$ lies in $\alpha$, then a priori it may be possible for other entries of $\alpha$ to play a role in an occurrence of $\pi$. However, now $x$ must be the first entry of $\rho$, and so the presence any other entry in $\alpha$ (other than $a$) in $\rho$ would force $\rho$ to start with a descent. Hence $d$ substitutes for $x$, as again $d$ cannot play a role in in the occurrence of $\pi$ and $d$ and $x$ are split neither by value nor by position by any other entry involved in the occurrence of $\pi$.
	
\begin{figure}
	\minipage{0.5\textwidth}
		\begin{center}
			\begin{tikzpicture}[scale = .6]
				\draw[draw=none, fill=black!15] (2,0) rectangle ++(2,2);
				\draw[draw=none, fill=black!15] (0,2) rectangle ++(2,2);
				\draw[draw=none, fill=black!15] (2,4) rectangle ++(2,2);
				\draw[draw=none, fill=black!15] (4,2) rectangle ++(2,2);
				\draw[draw=none, fill=black!15] (0,0) rectangle ++(2,2);
				\draw[draw=none, fill=black!15] (4,4) rectangle ++(.8,2);
				\foreach \i in {2,4} {
					\draw (\i,0) -- (\i,6);
					\draw (0,\i) -- (6,\i);
				}
				\node at (3,3) {$\alpha$};
				\draw[fill=black] (2, 3) circle (.08) node[left] {$a$};
				\draw[fill=black] (2.2, 5.2) circle (.08) node[above] {$x$};
				\draw[fill=black] (4.8,5) circle (.08) node[right] {$d$};
				\draw[dashed] (2.2,3) -- (2.2, 5.2) -- (4.8, 5.2);
			\end{tikzpicture}
		\end{center}
		\caption{The diagram corresponding to \emph{Case 1} in the proof of Theorem~\ref{theorem:bonds-inc}.\ \\}
		\label{figure:1plus-1}
	\endminipage\hfill
	\minipage{0.5\textwidth}
		\begin{center}
			\begin{tikzpicture}[scale = .6]
				\draw[draw=none, fill=black!15] (2,0) rectangle ++(2,2);
				\draw[draw=none, fill=black!15] (0,2) rectangle ++(2,2);
				\draw[draw=none, fill=black!15] (2,4) rectangle ++(2,2);
				\draw[draw=none, fill=black!15] (4,2) rectangle ++(2,2);
				\draw[draw=none, fill=black!15] (0,0) rectangle ++(2,2);
				\draw[draw=none, fill=black!15] (4,4) rectangle ++(.8,2);
				\draw[draw=none, fill=black!15] (0,4) rectangle ++(2,1);
				\foreach \i in {2,4} {
					\draw (\i,0) -- (\i,6);
					\draw (0,\i) -- (6,\i);
				}
				\node at (3,3) {$\alpha$};
				\draw[fill=black] (2, 3) circle (.08) node[left] {$a$};
				\draw[fill=black] (2.2, 5.2) circle (.08) node[above] {$x$};
				\draw[fill=black] (4.8,5) circle (.08) node[right] {$d$};
				\draw[dashed] (2.2,3) -- (2.2, 5.2) -- (4.8, 5.2);
			\end{tikzpicture}
		\end{center}
		\caption{The diagram corresponding to the first part of \emph{Case 2} in the proof of Theorem~\ref{theorem:bonds-inc}.}
		\label{figure:1plus-2}
	\endminipage
\end{figure}

	\bigskip
	
	\emph{Case 2: $\rho$ has no increasing bond}
	
	Let $a$ be the leftmost entry of $\alpha$ and let $d$ be the leftmost entry of $\delta$. We consider two separate cases: either $d$ is lower in value than all entries that lie in $\gamma$, or else there is some entry $c \in \gamma$ which is lower in value than $d$.
	
	First assume that $d$ is lower in value than all entries that lie in $\gamma$. Form $\omega^+$ by inserting an entry $x$ just to the right of $a$ and just above $d$, as in Figure~\ref{figure:1plus-2}. Note that this is the same placement as in \emph{Case 1}. As in the first part, if an occurrence of $\pi$ in $\omega^+$ contained any entry of $\alpha$ other than $a$, this would violate the assumption that $\rho$ started with an ascent. Moreover, the $1$ of $\pi$ cannot lie in $\gamma$, as there is no entry of $\gamma$ lower than $x$. This completes the proof under this assumption.
	
	Now assume that there exists an entry $c \in \gamma$ which is lower in value than $d$. Here we form $\omega^+$ in a different way, by placing the new entry $x$ just below $d$ instead of just above $d$. See Figure~\ref{figure:1plus-3} for a diagram of this placement. Suppose $\omega^+$ contains an occurrence of $\pi$. Then, $x$ must play a role in the $\rho$ part of $\pi$, otherwise $\rho \leq \delta$. We proceed as in \emph{Case 1}. If the $1$ of $\pi$ is in $\gamma$, then both $x$ and $d$ cannot be involved as they would form an increasing bond. Since no entry of $\alpha$ can be involved, $d$ substitutes for $x$. Otherwise, if the $1$ of $\pi$ is $a$, then $x$ is the second entry in the occurrence of $\pi$ and the third entry lies above and to the right of $d$ (because $\rho$ starts with an ascent). So, again, $d$ substitutes for $x$. This completes the proof of \emph{Case 2}.	
\begin{figure}
	\minipage{0.5\textwidth}
		\begin{center}
			\begin{tikzpicture}[scale = .6]
				\draw[draw=none, fill=black!15] (2,0) rectangle ++(2,2);
				\draw[draw=none, fill=black!15] (0,2) rectangle ++(2,2);
				\draw[draw=none, fill=black!15] (2,4) rectangle ++(2,2);
				\draw[draw=none, fill=black!15] (4,2) rectangle ++(2,2);
				\draw[draw=none, fill=black!15] (0,0) rectangle ++(2,2);
				\draw[draw=none, fill=black!15] (4,4) rectangle ++(.8,2);
				\foreach \i in {2,4} {
					\draw (\i,0) -- (\i,6);
					\draw (0,\i) -- (6,\i);
				}
				\node at (3,3) {$\alpha$};
				\draw[fill=black] (2, 3) circle (.08) node[left] {$a$};
				\draw[fill=black] (2.2, 4.8) circle (.08) node[above] {$x$};
				\draw[fill=black] (4.8,5) circle (.08) node[right] {$d$};
				\draw[fill=black] (1,4.5) circle (.08) node[left] {$c$};
				\draw[dashed] (2.2,3) -- (2.2, 4.8) -- (4.8, 4.8);
			\end{tikzpicture}
		\end{center}
		\caption{The diagram corresponding to the second part of \emph{Case 2} in the proof of Theorem~\ref{theorem:bonds-inc}.}
		\label{figure:1plus-3}
	\endminipage\hfill
	\minipage{0.5\textwidth}
		\begin{center}
			\begin{tikzpicture}[scale = .6]
				\draw[draw=none, fill=black!15] (2,0) rectangle ++(2,2);
				\draw[draw=none, fill=black!15] (0,2) rectangle ++(2,2);
				\draw[draw=none, fill=black!15] (2,4) rectangle ++(2,2);
				\draw[draw=none, fill=black!15] (4,2) rectangle ++(2,2);
				\draw[draw=none, fill=black!15] (0,0) rectangle ++(2,2);
				\draw[draw=none, fill=black!15] (4,1.5) rectangle ++(2,.5);
				\foreach \i in {2,4} {
					\draw (\i,0) -- (\i,6);
					\draw (0,\i) -- (6,\i);
				}
				\draw[fill] (5,1.5) circle (.08) node [below] {$e$};
				\draw[fill] (3,2) circle (.08) node [below] {$a$};
				\draw[fill] (5.2, 2.2) circle (.08) node [above] {$x$};
				\draw[dashed] (3,2.2)--(5.2,2.2)--(5.2,1.5);
				\node at (3,3) {$\alpha$};
			\end{tikzpicture}
		\end{center}
		\caption{The diagram corresponding to an initial placement of $x$ in Theorem~\ref{theorem:bonds-dec1}.}
		\label{figure:1plus-4}
	\endminipage
\end{figure}
\end{proof}


The above theorem handles all cases in which $\pi = 1 \oplus \rho$, where $\rho$ starts with an ascent and does not simultaneously have both kinds of bonds.  We next handle the case in which $\rho$ starts with a descent and has no increasing bond. For convenience, we say that $\pi$ satisfies condition $(\ddagger)$ if:
	\[\text{there is at least one entry to the right of 2 that is less than the leftmost entry of $\rho$ \tag{$\ddagger$}}\]

\begin{theorem}
	\label{theorem:bonds-dec1}
	Suppose that $\pi$ satisfies condition $(\ddagger)$, starts with a descent and has no increasing bond. Then, $\Av(\pi)$ is not deflatable.
\end{theorem}
\begin{proof}
	Note that $\pi^{-1} = 1 \oplus \rho^{-1}$. If $\rho^{-1}$ starts with an ascent, then we can appeal to the previous cases (as $\Av(\pi)$ is deflatable if and only if $\Av(\pi^{-1})$ is). Thus, we can assume that $\rho^{-1}$ starts with a descent. In terms of $\pi$, this implies that $3$ precedes $2$. If $2$ is the third entry of $\pi$, it follows that the first three entries of $\pi$ are $132$, and since $|\pi| > 3$, this implies that $\pi$ is a three-component sum, and thus is handled by Theorem~\ref{theorem:three-comp-sum}. Therefore, we can assume that $\pi$ has at least three entries preceding $2$ (at least two of which are part of $\rho$).
	
	Let $\omega \in \Av(\pi)$ be indecomposable, with all the usual additional assumptions. Suppose that $\epsilon$ is non-empty. Let $e$ be the topmost entry in $\epsilon$ and let $a$ be the bottommost entry in $\alpha$. Form $\omega^+$ by inserting an entry $x$ just to the right of $e$ and just above $a$ as in Figure~\ref{figure:1plus-4}. Suppose this introduces an occurrence of $\pi$. It follows that $x$ plays a role in the $\rho$ part of $\pi$. If the $1$ of $\pi$ is in $\epsilon$ then the $e$ cannot be involved, otherwise $e$ and $x$ form an increasing bond. This would allow $e$ to substitute for $x$. So, the $1$ must be in $\alpha$, and in fact the only possibility is that $a$ plays the role of the $1$.
  
	Since $a$ plays the role of the $1$, the role of the $2$ (the least element of $\rho$) must be played by $x$. If an occurrence of $\pi$ is created, condition $(\ddagger)$ forces all entries of $\pi$ other than $1$ and $2$ to be in $\delta$. Pick the leftmost (lexicographically least by position) possibilities for these entries of $\pi$. An example is given in Figure~\ref{figure:1plus-5} with $\pi = 153264$.
	
\begin{figure}
	\minipage{0.5\textwidth}
		\begin{center}
			\begin{tikzpicture}[scale = .6]
				\draw[draw=none, fill=black!15] (2,0) rectangle ++(2,2);
				\draw[draw=none, fill=black!15] (0,2) rectangle ++(2,2);
				\draw[draw=none, fill=black!15] (2,4) rectangle ++(2,2);
				\draw[draw=none, fill=black!15] (4,2) rectangle ++(2,2);
				\draw[draw=none, fill=black!15] (0,0) rectangle ++(2,2);
				\draw[draw=none, fill=black!15] (4,1.5) rectangle ++(2,.5);
				\foreach \i in {2,4} {
					\draw (\i,0) -- (\i,6);
					\draw (0,\i) -- (6,\i);
				}
				\draw[fill]  (4.4,5) circle (.08) node [above] {$5$};
				\draw[fill]  (4.8,4.4) circle (.08) node [above] {$3$};
				\draw[fill]  (5.4,5.3) circle (.08) node [above] {$6$};
				\draw[fill]  (5.8,4.7) circle (.08) node [above] {$4$};
				\draw[fill] (5,1.5) circle (.08) node [below] {$e$};
				\draw[fill] (3,2) circle (.08) node [below] {$a$} node [left=5pt, above] {$1$};
				\draw[fill] (5.2, 2.2) circle (.08) node [above=2pt, right] {$x$} node [above] {$2$};
				\draw[dashed] (3,2.2)--(5.2,2.2)--(5.2,1.5);
				\node at (3,3) {$\alpha$};
			\end{tikzpicture}
		\end{center}
		\caption{A diagram showing a possible occurrence of $\pi$ in Theorem~\ref{theorem:bonds-dec1}.\ \\}
		\label{figure:1plus-5}
	\endminipage\hfill
	\minipage{0.5\textwidth}
		\begin{center}
			\begin{tikzpicture}[scale = .6]
				\draw[draw=none, fill=black!15] (2,0) rectangle ++(2,2);
				\draw[draw=none, fill=black!15] (0,2) rectangle ++(2,2);
				\draw[draw=none, fill=black!15] (2,4) rectangle ++(2,2);
				\draw[draw=none, fill=black!15] (4,2) rectangle ++(2,2);
				\draw[draw=none, fill=black!15] (0,0) rectangle ++(2,2);
				\draw[draw=none, fill=black!15] (4,1.5) rectangle ++(2,.5);
				\draw[draw=none, fill=black!15] (4,0) rectangle ++(.5,2);
				\draw[draw=none, fill=black!15] (4,4) rectangle ++(.4,2);
				\foreach \i in {2,4} {
					\draw (\i,0) -- (\i,6);
					\draw (0,\i) -- (6,\i);
				}     
				\draw[fill]  (4.4,5) circle (.08) node [above] {$$};
				\draw[fill] (5,1.5) circle (.08) node [below] {$e$};
				\draw[fill] (3,2) circle (.08) node [below] {$a$} node [above] {$$};
				\draw[fill] (4.6, 2.2) circle (.08) node [above=2pt, right] {$x'$} node [above] {$$};
				\draw[dashed] (3,2.2)--(4.6,2.2)--(4.6,5);
				\node at (3,3) {$\alpha$};
			\end{tikzpicture}
		\end{center}
		\caption{The new placement of a splitting entry in the proof of Theorem~\ref{theorem:bonds-dec1}.}
		\label{figure:1plus-6}
	\endminipage
\end{figure}
	
	In this case, there can be no entry in $\epsilon$ which lies to the left of the entry which played the role of the first entry of $\rho$ (in the example above, $\epsilon$ has no entry which lies to the left of the entry marked ``$5$''): otherwise that entry could would have played the $1$ in an occurrence of $\pi$ in $\omega$ which involved $e$ as the $2$ and the same remaining entries in $\delta$. Now, place a new splitting entry $x'$ just above $a$ and just to the right of the leftmost entry of $\delta$ (which may or may not be one of the entries in the occurrence of $\pi$), as in Figure~\ref{figure:1plus-6}.
  
Now, we have forced $a$ to play the role of the $1$ in any occurrence of $\pi$ which involves $x'$. Again, $x'$ must play the role of the $2$ and the remaining entries of $\pi$ would have to be in $\delta$. However, since we have already shown that three entries of $\pi$ must precede $2$, this is impossible. Therefore, the introduction of $x'$ does not introduce an occurrence of $\pi$, and indeed $\omega^+ \in \Av(\pi)$.

If $\epsilon$ were actually empty, then the splitting created by $x'$ works for the same reason.
\end{proof}

We continue under the assumption that $\pi = 1 \oplus \rho$ where $\rho$ starts with a descent and has no increasing bond. Assume that $\pi$ does not satisfy condition $(\ddagger)$; that is, assume that the first entry of $\rho$ is less than every entry to the right of the entry $2$. Moreover, we can assume that $\pi^{-1}$ also fails $(\ddagger)$. In terms of $\pi$, this translates to the property that the entry $2$ precedes every entry which has value greater than the first entry of $\rho$.

After a little inspection, one can see if $\pi$ and $\pi^{-1}$ both fail condition $(\ddagger)$, then either $\pi$ is a sum of three or more components or $\pi$ has the form $1n\cdots 2$ (by this, we do not mean that $\rho$ is decreasing, just that $\rho$ starts with its biggest entry and ends with its smallest entry). The former case is already proved, so we only need to prove the latter.

\begin{theorem}
	\label{theorem:bonds-dec2}
	Suppose that $\pi = 1 \oplus \rho$ is of the form $1n \cdots 2$ and $\rho$ has no increasing bond. Then $\Av(\pi)$ is not deflatable.
\end{theorem}
\begin{proof}
	Let $a$ be the topmost entry of $\alpha$, let $d$ be the leftmost entry of $\delta$, and let $e$ be the leftmost entry of $\epsilon$ (if it exists). We proceed in two cases. Suppose first that $\epsilon$ is empty or that $d$ precedes $e$. Place a splitting entry $x$ just to the right of $d$ and just below $a$, as in Figure~\ref{figure:1plus-7}.
	
\begin{figure}
	\minipage{0.5\textwidth}
		\begin{center}
			\begin{tikzpicture}[scale = .6]
				\draw[draw=none, fill=black!15] (2,0) rectangle ++(2,2);
				\draw[draw=none, fill=black!15] (0,2) rectangle ++(2,2);
				\draw[draw=none, fill=black!15] (2,4) rectangle ++(2,2);
				\draw[draw=none, fill=black!15] (4,2) rectangle ++(2,2);
				\draw[draw=none, fill=black!15] (0,0) rectangle ++(2,2);
				\draw[draw=none, fill=black!15] (4,0) rectangle ++(.8,2);
				\draw[draw=none, fill=black!15] (4,4) rectangle ++(.4,2);
				\foreach \i in {2,4} {
					\draw (\i,0) -- (\i,6);
					\draw (0,\i) -- (6,\i);
				}
				\draw[fill]  (4.4,5) circle (.08) node [above = 4pt, right] {$d$};
				\draw[fill]  (4.8,1) circle (.08) node [right] {$e$};
				\draw[fill] (3,4) circle (.08) node [above] {$a$};
				\draw[fill] (4.6, 3.8) circle (.08) node [below] {$x$};
				\draw[dashed, thick] (3,3.8)--(4.6,3.8)--(4.6,5);
				\node at (3,3) {$\alpha$};
			\end{tikzpicture}
		\end{center}
		\caption{A diagram corresponding to $\omega^+$ in the first case of Theorem~\ref{theorem:bonds-dec2}.}
		\label{figure:1plus-7}
	\endminipage\hfill
	\minipage{0.5\textwidth}
		\begin{center}
			\begin{tikzpicture}[scale = .6]
				\draw[draw=none, fill=black!15] (2,0) rectangle ++(2,2);
				\draw[draw=none, fill=black!15] (0,2) rectangle ++(2,2);
				\draw[draw=none, fill=black!15] (2,4) rectangle ++(2,2);
				\draw[draw=none, fill=black!15] (4,2) rectangle ++(2,2);
				\draw[draw=none, fill=black!15] (0,0) rectangle ++(2,2);
				\draw[draw=none, fill=black!15] (4,0) rectangle ++(.4,2);
				\draw[draw=none, fill=black!15] (4,4) rectangle ++(.6,2);
				\draw[draw=none, fill=black!15] (4,4) rectangle ++(2,.4);
				\foreach \i in {2,4} {
					\draw (\i,0) -- (\i,6);
					\draw (0,\i) -- (6,\i);
				}
				\draw[fill]  (4.6,5) circle (.08) node [above=3pt, right] {$d$};
				\draw[fill]  (4.4,1) circle (.08) node [right] {$e$};
				\draw[fill]  (5.3,4.4) circle (.08) node [right=3pt,above] {$d'$};
				\draw[fill] (3,4) circle (.08) node [above] {$a$};
				\draw[fill] (5.1, 3.8) circle (.08) node [below] {$x$};
				\draw[dashed, thick] (3,3.8)--(5.1,3.8)--(5.1,4.4);
				\node at (3,3) {$\alpha$};
			\end{tikzpicture}
		\end{center}
		\caption{A diagram corresponding to $\omega^+$ in the second case of Theorem~\ref{theorem:bonds-dec2}.}
		\label{figure:1plus-8}
	\endminipage
\end{figure}

In this case, $x$ must play a role in the $\rho$. Therefore the $1$ of any occurrence of $\pi$ is in $\alpha$ (and is not $a$). If $x$ is not the $2$, then there is no place for the $2$. Only $a$ and $d$ can be entries of $\pi$ other than the $1$ or $2$. Therefore, $\pi = 132$ (a known case) or $\pi = 1342$ (not of the form $1n\cdots 2$). Hence, this case is complete.

Suppose instead that $e$ precedes $d$. Let $d'$ be the bottommost entry in $\delta$ (it is possible that $d=d'$). Place the splitting entry $x$ just below $a$ and just to the left of $d'$, as in Figure~\ref{figure:1plus-8}. Suppose there is an occurrence of $\pi$. Then, $x$ must play a role in the $\rho$ of such an occurrence.

If the $1$ of this occurrence is in $\alpha$ and $a$ is not involved, then since $x$ and $d'$ cannot both be involved, $d'$ can substitute for $x$. Thus $a$ must be involved. If $x$ is not the $2$, then there is no place to put the $2$. This forces $x$ to be the $2$ and $a$ to be the $3$. However, unless $\pi = 132$, there is no place now for the biggest entry of $\pi$.

If the $1$ of this occurrence is in $\epsilon$, the we can substitute $d'$ for $x$. This completes the second case, and the proof.
\end{proof}

We have now disposed of the case where $\pi = 1 \oplus \rho$ where $\rho$ starts with a descent and has no increasing bond. The last case we handle is when $\pi = 1 \oplus \rho$ where $\rho$ starts with a descent and has no decreasing bond. As before, we can further assume that $\rho^{-1}$ starts with a descent, i.e., that $3$ precedes $2$ in $\pi$. The two proofs below largely mirror the previous two proofs, with some small changes in the easy cases.

\begin{theorem}
	\label{theorem:bonds-dec3}
	Suppose that $\pi = 1 \oplus \rho$ satisfies condition $(\ddagger)$, and that $\rho$ starts with a descent and has no decreasing bond. Then $\Av(\pi)$ is not deflatable.
\end{theorem}
\begin{proof}
	If $\epsilon$ is empty we can use the same splitting construction as in the proof of Theorem~\ref{theorem:bonds-dec1}. So, assume now that $\epsilon$ is non-empty. Let $a$ be the bottommost entry of $\alpha$, let $d$ be the leftmost entry of $\delta$, and let $e$ be the topmost entry of $\epsilon$. We handle two separate cases. First assume that $e$ precedes $d$. In this case, we must have that $e$ is not also the leftmost entry of $\epsilon$, or else $\alpha$ is not a maximal interval. Place a splitting entry $x$ just above $a$ and just to the left of $e$, as shown in Figure~\ref{figure:1plus-9}.

\begin{figure}
	\minipage{0.5\textwidth}
		\begin{center}
			\begin{tikzpicture}[scale = .6]
				\draw[draw=none, fill=black!15] (2,0) rectangle ++(2,2);
				\draw[draw=none, fill=black!15] (0,2) rectangle ++(2,2);
				\draw[draw=none, fill=black!15] (2,4) rectangle ++(2,2);
				\draw[draw=none, fill=black!15] (4,2) rectangle ++(2,2);
				\draw[draw=none, fill=black!15] (0,0) rectangle ++(2,2);
				\draw[draw=none, fill=black!15] (4,0) rectangle ++(.4,2);
				\draw[draw=none, fill=black!15] (4,4) rectangle ++(1.5,2);
				\draw[draw=none, fill=black!15] (4,1.5) rectangle ++(2,.5);
				\foreach \i in {2,4} {
					\draw (\i,0) -- (\i,6);
					\draw (0,\i) -- (6,\i);
				}
				\draw[fill]  (5.5,5) circle (.08) node [right] {$d$};
				\draw[fill]  (4.4,.7) circle (.08) node [above] {$$};
				\draw[fill]  (5.2,1.5) circle (.08) node [below] {$e$};
				\draw[fill] (3,2) circle (.08) node [below] {$a$};
				\draw[fill] (5, 2.2) circle (.08) node [above] {$x$};
				\draw[dashed, thick] (3,2.2)--(5,2.2)--(5,1.5);
				\node at (3,3) {$\alpha$};
			\end{tikzpicture}
		\end{center}
		\caption{A diagram corresponding to $\omega^+$ in the first case of Theorem~\ref{theorem:bonds-dec3}.}
		\label{figure:1plus-9}
	\endminipage\hfill
	\minipage{0.5\textwidth}
		\begin{center}
			\begin{tikzpicture}[scale = .6]
				\draw[draw=none, fill=black!15] (2,0) rectangle ++(2,2);
				\draw[draw=none, fill=black!15] (0,2) rectangle ++(2,2);
				\draw[draw=none, fill=black!15] (2,4) rectangle ++(2,2);
				\draw[draw=none, fill=black!15] (4,2) rectangle ++(2,2);
				\draw[draw=none, fill=black!15] (0,0) rectangle ++(2,2);
				\draw[draw=none, fill=black!15] (4,1.5) rectangle ++(2,.5);
				\draw[draw=none, fill=black!15] (4,4) rectangle ++(.3,2);
				\foreach \i in {2,4} {
					\draw (\i,0) -- (\i,6);
					\draw (0,\i) -- (6,\i);
				}
				\draw[fill] (4.3,5) circle (.08) node[right=4pt, above] {$d$};
				\draw[fill] (5,1.5) circle (.08) node [below] {$e$};
				\draw[fill] (3,2) circle (.08) node [below] {$a$};
				\draw[fill] (4.8, 2.2) circle (.08) node [above] {$x$};
				\draw[dashed, thick] (3,2.2)--(4.8,2.2)--(4.8,1.5);
				\node at (3,3) {$\alpha$};
			\end{tikzpicture}
		\end{center}
		\caption{A diagram corresponding to $\omega^+$ in the second case of Theorem~\ref{theorem:bonds-dec3}.}
		\label{figure:1plus-10}
	\endminipage
\end{figure}

If there is an occurrence of $\pi$ in $\omega^+$, then the $1$ of $\pi$ either lies in $\alpha$ or $\epsilon$. If the $1$ of $\pi$ lies in $\epsilon$, then we can substitute $e$ for $x$ since both cannot be involved. If the $1$ is in $\alpha$, then $a$ is the $1$ and $x$ is the $2$. Since $\pi$ satisfies condition $(\ddagger)$, all other entries must be in $\delta$, but then the first two entries $\pi$ are $12$, a contradiction.
  
Suppose instead that $d$ precedes $e$. Place a splitting entry $x$ just above $a$ and just to the left of $e$, as in Figure~\ref{figure:1plus-10}. Again, if the $1$ of an occurrence of $\pi$ is in $\epsilon$, then we can substitute $e$ for $x$. Thus, the $1$ is in $\alpha$, and $a=1$ and $x=2$. Since $\pi$ satisfies condition $(\ddagger)$, all other entries of $\pi$ lie in $\delta$. Define $\hat{\pi} = \pi \smallsetminus \{a,x\}$ and let $\hat{d}$ be the leftmost entry of $\hat{\pi}$. It is possible that $\hat{d} = d$. See Figure~\ref{figure:1plus-11}.
  
\begin{figure}
	\minipage{0.5\textwidth}
		\begin{center}
			\begin{tikzpicture}[scale = .6]
				\draw[draw=none, fill=black!15] (2,0) rectangle ++(2,2);
				\draw[draw=none, fill=black!15] (0,2) rectangle ++(2,2);
				\draw[draw=none, fill=black!15] (2,4) rectangle ++(2,2);
				\draw[draw=none, fill=black!15] (4,2) rectangle ++(2,2);
				\draw[draw=none, fill=black!15] (0,0) rectangle ++(2,2);
				\draw[draw=none, fill=black!15] (4,1.5) rectangle ++(2,.5);
				\draw[draw=none, fill=black!15] (4,4) rectangle ++(.3,2);
				\foreach \i in {2,4} {
					\draw (\i,0) -- (\i,6);
					\draw (0,\i) -- (6,\i);
				}
				\draw[thick] (4.7,4.2) rectangle (5.7, 5.8) node[right=4pt, above=0pt, midway] {$\hat{\pi}$};
				\draw[fill] (4.3,5) circle (.08) node[right=3pt, above] {$d$};
				\draw[fill] (4.7,4.7) circle (.08) node[right] {$\hat{d}$};
				\draw[fill] (5,1.5) circle (.08) node [below] {$e$};
				\draw[fill] (3,2) circle (.08) node [below] {$a$};
				\draw[fill] (4.8, 2.2) circle (.08) node [above] {$x$};
				\draw[dashed, thick] (3,2.2)--(4.8,2.2)--(4.8,1.5);
				\node at (3,3) {$\alpha$};
			\end{tikzpicture}
		\end{center}
		\caption{A diagram corresponding to $\omega^+$ in the second case of Theorem~\ref{theorem:bonds-dec3}.}
		\label{figure:1plus-11}
	\endminipage\hfill
	\minipage{0.5\textwidth}
		\begin{center}
			\begin{tikzpicture}[scale = .6]
				\draw[draw=none, fill=black!15] (2,0) rectangle ++(2,2);
				\draw[draw=none, fill=black!15] (0,2) rectangle ++(2,2);
				\draw[draw=none, fill=black!15] (2,4) rectangle ++(2,2);
				\draw[draw=none, fill=black!15] (4,2) rectangle ++(2,2);
				\draw[draw=none, fill=black!15] (0,0) rectangle ++(2,2);
				\draw[draw=none, fill=black!15] (4,1.5) rectangle ++(2,.5);
				\draw[draw=none, fill=black!15] (4,4) rectangle ++(.3,2);
				\draw[draw=none, fill=black!15] (4,0) rectangle ++(.7,2);
				\foreach \i in {2,4} {
					\draw (\i,0) -- (\i,6);
					\draw (0,\i) -- (6,\i);
				}
				\draw[thick] (4.7,4.2) rectangle (5.7, 5.8) node[right=4pt, above=0pt, midway] {$\hat{\pi}$};
				\draw[fill] (4.3,5) circle (.08) node[right=3pt, above] {$d$};
				\draw[fill] (4.7,4.7) circle (.08) node[right] {$\hat{d}$};
				\draw[fill] (5,1.5) circle (.08) node [below] {$e$};
				\draw[fill] (3,2) circle (.08) node [below] {$a$};
				\draw[fill] (4.5, 2.2) circle (.08) node [right=6pt, above] {$x'$};
				\draw[dashed, thick] (3,2.2)--(4.5,2.2)--(4.5,5);
				\node at (3,3) {$\alpha$};
			\end{tikzpicture}
		\end{center}
		\caption{A diagram corresponding to $\omega^+$ in the second case of Theorem~\ref{theorem:bonds-dec3}.}
		\label{figure:1plus-12}
	\endminipage
\end{figure}

If there is an entry $z$ in $\epsilon$ that precedes $\hat{d}$, then $z$ and $e$ can together play the same roles as $a$ and $x$, creating an occurrence of $\pi$ in $\omega$. If not, then place a splitting entry $x'$ just to the right of $d$ and just above $a$, as in Figure~\ref{figure:1plus-12}. If $x'$ creates another occurrence of $\pi$, then we must have $a=1$ and $x'=2$; this forces this entry of $\pi$ to be $2$. Since we know that $3$ precedes $2$, it follows that $\pi$ is a three-component sum, and we can appeal to Theorem~\ref{theorem:three-comp-sum}.
\end{proof}

Lastly, we consider the case in which $\pi$ fails condition $(\ddagger)$. As before, we may also assume that $\pi^{-1}$ fails condition $(\ddagger)$, leaving us only with the case $\pi = 1n\cdots 2$, where $\pi$ has no decreasing bond.

\begin{theorem}
	\label{theorem:bonds-dec4}
Let $\pi$ have the form $1n \cdots 2$, such that $\pi$ has no decreasing bond. Then $\Av(\pi)$ is not deflatable.
\end{theorem}
\begin{proof}
	Let $a$ be the topmost entry in $\alpha$ and let $d$ be the bottommost entry of $\delta$. Place a splitting point $x$ just below $a$ and just to the right of $d$, as in Figure~\ref{figure:1plus-13}.
	
\begin{figure}
	\minipage{0.5\textwidth}
		\begin{center}
			\begin{tikzpicture}[scale = .6]
				\draw[draw=none, fill=black!15] (2,0) rectangle ++(2,2);
				\draw[draw=none, fill=black!15] (0,2) rectangle ++(2,2);
				\draw[draw=none, fill=black!15] (2,4) rectangle ++(2,2);
				\draw[draw=none, fill=black!15] (4,2) rectangle ++(2,2);
				\draw[draw=none, fill=black!15] (0,0) rectangle ++(2,2);
				\draw[draw=none, fill=black!15] (4,4) rectangle ++(2,.4);
				\foreach \i in {2,4} {
					\draw (\i,0) -- (\i,6);
					\draw (0,\i) -- (6,\i);
				}   
				\draw[fill]  (5.3,4.4) circle (.08) node [above] {$d$};
				\draw[fill] (3,4) circle (.08) node [above] {$a$};
				\draw[fill] (5.5, 3.8) circle (.08) node [below] {$x$};
				\draw[dashed, thick] (3,3.8)--(5.5,3.8)--(5.5,4.4);
				\node at (3,3) {$\alpha$};
			\end{tikzpicture} 
			\caption{A diagram corresponding to $\omega^+$ in Theorem~\ref{theorem:bonds-dec4}.}
			\label{figure:1plus-13}
		\end{center}
	\endminipage\hfill
	\minipage{0.5\textwidth}
		\begin{center}
			\begin{tikzpicture}[scale = .6]
				\draw[draw=none, fill=black!15] (2,0) rectangle ++(2,2);
				\draw[draw=none, fill=black!15] (0,2) rectangle ++(2,2);
				\draw[draw=none, fill=black!15] (2,4) rectangle ++(2,2);
				\draw[draw=none, fill=black!15] (4,2) rectangle ++(2,2.7);
				\draw[draw=none, fill=black!15] (0,0) rectangle ++(2,2);
				\foreach \i in {2,4} {
					\draw (\i,0) -- (\i,6);
					\draw (0,\i) -- (6,\i);
				}
				\draw[fill=black] (2.2, 4.5) circle (.08) node[above] {$x$};
				\draw[fill=black] (5, 4.7) circle (.08) node[above] {$d$};
				\draw[fill=black] (2, 3) circle (.08) node[left] {$a$};
				\draw[fill=black] (.5, 4.2) circle (.08) node[left] {$c$};
				\draw[dashed] (2.2,3) -- (2.2, 4.5) -- (5, 4.5);
				\node at (3,3) {$\alpha$};
			\end{tikzpicture}
		\end{center}
		\caption{A diagram of the first case in the proof of Theorem~\ref{theorem:1...2}}
		\label{figure:1x2-1}
	\endminipage
\end{figure}

If there is an occurrence of $\pi$ and $a$ is not involved, then $d$ can substitute for $x$, as both cannot be involved simultaneously. Hence, $a$ must be involved, and there is no entry of $\pi$ in $\epsilon$. Moreover, $a$ cannot be the $1$ in an occurrence of $\pi$, since $x$ must be involved. Therefore, $a$ must be the $n$, forcing $x$ to be the $n-1$. It follows there is no allowed location for any other entries which play a role in $\pi$. Hence $\pi=132$. As $132$ is known to be deflatable, this is a contradiction.
\end{proof}

Theorems~\ref{theorem:three-comp-sum}-\ref{theorem:bonds-dec4} tell us that if a principal class $\Av(\pi)$ is deflatable for sum-decomposable $\pi$, then $\pi$ must have the form $1 \oplus \rho$, where $\rho$ is sum-indecomposable and contains both an increasing and decreasing bond. However, as the next theorem shows, it is still possible that $\pi$ can have these properties and $\Av(\pi)$ still be non-deflatable.


\begin{theorem}
	\label{theorem:1...2}
	Let $\pi = 1 \oplus \rho$ for $\rho$ of the form $x\cdots1$ with $x \neq 2$ and $x \neq |\rho|$, i.e., $\pi = 1z\cdots2$ with $z \neq 3$ and $z \neq |\pi|$. Then, $\Av(\pi)$ is not deflatable.
\end{theorem}
\begin{proof}
	Let $\omega \in \Av(\pi)$ be indecomposable (and, as always, not simple). Let $a$ be the leftmost entry of $\alpha$ and let $d$ be the bottommost entry in $\delta$. We proceed in two separate cases.

	First suppose that there is an entry $c$ in $\gamma$ which lies below $d$. In this case, insert an entry $x$ just to the right of $a$ and just below $d$ (see Figure~\ref{figure:1x2-1}). Suppose that an occurrence of $\pi$ is created. If $x$ is the last entry of this $\pi$, then $d$ substitutes for $x$, contradicting the assumption that $\omega \in \Av(\pi)$. So suppose that $x$ is not the last entry of the occurrence of $\pi$.

	If $a$ is the $1$ of $\pi$, then $x$ must be the second entry of $\pi$, which is not the biggest entry of $\pi$. Therefore, $\pi$ must contain an entry larger than $x$ to the right, and hence larger than $d$ and to the right of $\alpha$. However, the final entry of $\pi$ must be smaller than $x$ and larger than $a$, and there is no such entry. If $\pi$ starts with an entry in $\gamma$, then since $x$ is not the last entry of an occurrence of $\pi$, there is no place for the last entry of $\pi$ anywhere. This completes the first case.
	
	Now assume otherwise, that no entry in $\gamma$ lies below $d$, as in Figure~\ref{figure:1x2-2}. Place a new entry $x$ just to the right of $a$ and just above $d$. An occurrence of $\pi$ cannot involve $x$ as the $1$, since otherwise the remainder of the pattern would lie in the upper right quadrant, and so the entry $a$ must play the role of the $1$.
	
	Thus, $x$ plays the role of the first entry of $\rho$, which is neither the biggest nor the smallest entry in $\rho$. However, the last entry of $\pi$ is $2$, and now we see that there is no place for the $2$.	
\begin{figure}
	\minipage{0.5\textwidth}
		\begin{center}
			\begin{tikzpicture}[scale = .6]
				\draw[draw=none, fill=black!15] (2,0) rectangle ++(2,2);
				\draw[draw=none, fill=black!15] (0,2) rectangle ++(2,2.9);
				\draw[draw=none, fill=black!15] (2,4) rectangle ++(2,2);
				\draw[draw=none, fill=black!15] (4,2) rectangle ++(2,2.7);
				\draw[draw=none, fill=black!15] (0,0) rectangle ++(2,2);
				\foreach \i in {2,4} {
					\draw (\i,0) -- (\i,6);
					\draw (0,\i) -- (6,\i);
				}
				\draw[fill=black] (2.2, 4.9) circle (.08) node[label=above:$x$] (x) {};
				\draw[fill=black] (5, 4.7) circle (.08) node[label=above:$d$] (c) {};
				\draw[fill=black] (2, 3) circle (.08) node[label=left:$a$] (a) {};
				\draw[dashed] (x) -- ++(2.8,0);
				\draw[dashed] (x) -- ++(0,-1.8);
				\node at (3,3) {$\alpha$};
			\end{tikzpicture}
		\end{center}
		\caption{A diagram of the second case in the proof of Theorem~\ref{theorem:1...2}}
		\label{figure:1x2-2}
	\endminipage\hfill
	\minipage{0.5\textwidth}
		\begin{center}
		\begin{tikzpicture}[scale=.3]
			\filldraw[dark-gray](0,3) rectangle (1,4);
			\filldraw[dark-gray](1,3) rectangle (2,4);
			\filldraw[dark-gray](1,6) rectangle (2,7);
			\filldraw[dark-gray](1,7) rectangle (2,8);
			\filldraw[dark-gray](2,3) rectangle (3,4);
			\filldraw[dark-gray](2,6) rectangle (3,7);
			\filldraw[dark-gray](2,7) rectangle (3,8);
			\filldraw[dark-gray](3,2) rectangle (4,3);
			\filldraw[dark-gray](3,3) rectangle (4,4);
			\filldraw[dark-gray](4,0) rectangle (5,1);
			\filldraw[dark-gray](4,1) rectangle (5,2);
			\filldraw[dark-gray](5,0) rectangle (6,1);
			\filldraw[dark-gray](5,1) rectangle (6,2);
			\filldraw[dark-gray](5,5) rectangle (6,6);
			\filldraw[dark-gray](5,6) rectangle (6,7);
			\filldraw[dark-gray](5,7) rectangle (6,8);
			\filldraw[dark-gray](5,8) rectangle (6,9);
			\filldraw[dark-gray](6,5) rectangle (7,6);
			\filldraw[dark-gray](7,3) rectangle (8,4);
			\filldraw[dark-gray](7,4) rectangle (8,5);
			\filldraw[dark-gray](8,3) rectangle (9,4);
			\filldraw[dark-gray](8,4) rectangle (9,5);

			\draw[black, fill=black] (1,2) circle (0.2);
			\draw[black, fill=black] (2,5) circle (0.2);
			\draw[black, fill=black] (3,1) circle (0.2);
			\draw[black, fill=black] (4,7) circle (0.2);
			\draw[black, fill=black] (5,3) circle (0.2);
			\draw[black, fill=black] (6,4) circle (0.2);
			\draw[black, fill=black] (7,8) circle (0.2);
			\draw[black, fill=black] (8,6) circle (0.2);

			\draw (0,0) grid (9,9);
		\end{tikzpicture}
		\caption{The permutation diagram of the permutation $25173486 \in \Av(251364)$.}
		\label{figure:25173486}
	\end{center}
	\endminipage
\end{figure}
\end{proof}

\section{Deflatable Permutation Classes}
\label{section:def}

Given the results of the previous section, one may wonder whether any principal classes are deflatable other than $\Av(12)$, $\Av(231)$ and their symmetries. For the larger group of finitely-based classes, the answer is clear: any class with finitely many simples (and infinitely many permutations) must be deflatable, and there are infinitely many such classes. Moreover, the results referred to in the introduction make use of the fact that many classes $\Av(\alpha, \beta)$ with $|\alpha| = |\beta| = 4$ turn out to be deflatable. In this section, we first provide a criterion by which we may prove deflatability of $\Av(\pi)$. We use this criterion to show examples of deflatable classes $\Av(\pi)$ for which $\pi$ is decomposable, simple, or neither.

In a deflatable class $\C$, there are permutations $\tau$ which cannot be extended to a simple permutation, i.e., there exists no simple $\sigma \in \C$ such that $\sigma \geq \tau$. Therefore, if we find a $\tau$ with this property in a class $\C$, it follows that $\C$ is deflatable. We call such a $\tau$ a \emph{witness of deflatability}.

In this section, we represent permutations by their diagrams, as shown in Section~\ref{section:introduction} and as produced by PermLab~\cite{albert:permlab}. A square in a permutation diagram is shaded gray if inserting an entry in that square would create a forbidden pattern. 

The lemma we now prove aids in finding witnesses of deflatability.

\begin{lemma}
	\label{lemma:witness-extension}
	Let $\omega \in \C$ contain a bond such that no entry may be placed in any square either horizontally or vertically between the two entries of the bond, with the possible exception of the four adjacent squares to the bond. 
	Then, $\omega$ cannot be extended to a simple permutation in $\C$.
\end{lemma}
\begin{proof}
The figure below gives an example of the configuration in question, where the diagonally shaded quadrants may contain any entries.
	
	
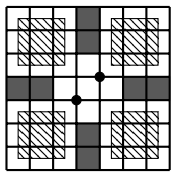
\begin{center}
		\begin{tikzpicture}[scale=.31]

			\filldraw[dark-gray] (3,0) rectangle (4,2);
			\filldraw[dark-gray] (5,3) rectangle (7,4);
			\filldraw[dark-gray] (3,5) rectangle (4,7);
			\filldraw[dark-gray] (0,3) rectangle (2,4);

			\draw[black, fill=black] (3,3) circle (0.2);
			\draw[black, fill=black] (4,4) circle (0.2);

			\draw[thick] (0,0) grid (7,7);
			
			\draw[pattern=north west lines, pattern color=black] (.5,.5) rectangle (2.5,2.5);
			\draw[pattern=north west lines, pattern color=black] (4.5,.5) rectangle (6.5,2.5);
			\draw[pattern=north west lines, pattern color=black] (.5,4.5) rectangle (2.5,6.5);
			\draw[pattern=north west lines, pattern color=black] (4.5,4.5) rectangle (6.5,6.5);
		\end{tikzpicture}
	\end{center}
			
	Let $\omega \in \C$ be as in the statement of the Lemma. Assume without loss of generality that the bond of interest is an increasing bond; the proof follows, mutatis mutandis, when the bond is a decreasing bond. Let $\omega^+ \in \C$ contain $\omega$, and fix an occurrence of $\omega$ in $\omega^+$. In the rest of the proof we refer to this occurrence as $\omega$. Let $\nu$ be the maximal box in $\omega^+$ which contains the two points from the bond of $\omega$, and which is cut by no other point of $\omega$. Note that we could replace these two points by any pair of points of pattern $12$ inside $\nu$ and still have an occurrence of $\omega$.
	
Since $\nu$ has at least one pair of increasing entries, we know that its skew-decomposition (which may have only one summand) has at least one non-trivial skew-indecomposable summand, which we will call $\theta$. Note that it is possible that $\theta = \nu$. We show that $\omega^+$ is not simple by showing that $\theta$ is an interval of $\omega^+$

Suppose toward a contradiction that $\omega^+$ has an entry $x$ which cuts $\theta$. Then, $x$ does not lie in $\nu$ because $\theta$ was chosen as an interval of $\nu$. So $x$ must lie in one of the four regions adjacent to $\nu$, and separated from $\nu$ by an element of $\omega$; without loss of generality, we assume that $x$ lies in the region above $\nu$. It follows that every entry of $\theta$ which lies to the left of $x$ is greater in value than every entry of $\theta$ which lies to the right of $x$; otherwise, $x$ would lie in the forbidden region defined by the embedded occurrence of $\omega$. This contradicts the assumption that $\theta$ is skew-indecomposable. Hence, $\theta$ is an interval of length greater than $1$ and thus $\omega^+$ is not simple.
\end{proof}

We can now proceed to identify a number of deflatable principal classes. For example, consider the diagram of the permutation $25173486 \in \Av(251364)$, as shown in Figure~\ref{figure:25173486}. By Lemma~\ref{lemma:witness-extension}, the permutation $25173486$ is a witness of deflatability for the class $\Av(251364)$, proving that $\Av(251364)$ is deflatable. We list below a sporadic collection of deflatable classes and witnesses which prove their deflatability. These witnesses were found through a mixture of computer search and ``by hand'' construction.

\begin{center}
	\begin{tabular}{||c|c||}
		\hline\hline
		Permutation Class & Witness of Deflatability\\
		\hline\hline
		$\Av(134652)$ & $6\;8\;9\;3\;4\;1\;10\;14\;7\;13\;5\;12\;11\;2$\\
		\hline
		$\Av(246135)$ & $4\;7\;2\;9\;11\;5\;6\;1\;10\;3\;8$\\
		\hline
		$\Av(246513)$ & $5\;9\;3\;11\;8\;2\;10\;6\;7\;1\;4$\\
		\hline
		$\Av(251364)$ & $2\;5\;1\;7\;3\;4\;8\;6$\\
		\hline
		$\Av(251463)$ & $2\;6\;1\;8\;4\;3\;7\;9\;5$\\
		\hline
		$\Av(254613)$ & $5\;9\;3\;11\;2\;8\;10\;6\;7\;1\;4$\\
		\hline
		$\Av(256413)$ & $4\;7\;9\;2\;10\;8\;5\;6\;1\;3$\\
		\hline
		$\Av(1523764)$ & $11\;18\;14\;16\;8\;19\;6\;7\;22\;13\;1\;10\;5\;24\;2\;3\;9\;17\;23\;4\;21\;20\;15\;12$\\
		\hline
		$\Av(2613475)$ & $2\;6\;1\;3\;9\;4\;5\;7\;10\;8$\\
		\hline
		$\Av(2631574)$ & $2\;6\;3\;1\;9\;5\;4\;8\;10\;7$\\
		\hline\hline
	\end{tabular}
\end{center}

One should first note that the classes $\Av(134652)$ and $\Av(1523764)$ are listed in the above table. That these classes are deflatable proves that, in fact, not all classes of the form $\Av(\pi)$ for decomposable $\pi$ are non-deflatable. Both classes consists of basis elements which have both increasing and decreasing bonds, in some sense justifying the care taken in the previous section when dealing with permutations which contained at most one type of bond.

Many of the other basis elements of classes in the list are simple. It is of particular interest that the class $\Av(246135)$ is deflatable, as it is a special type of simple permutation: a \emph{parallel alternation}. In fact, the classes $\Av(24681357)$, $\Av(2\;4\;6\;8\;10\;1\;3\;5\;7\;9)$, $\Av(2\;4\;6\;8\;10\;12\;1\;3\;5\;7\;9\;11)$, and $\Av(2\;4\;6\;8\;10\;12\;14\;1\;3\;5\;7\;9\;11\;13)$ are also deflatable, as shown by the witnesses
	\begin{center}
		\begin{tabular}{l}
			$5\;8\;11\;2\;13\;4\;14\;16\;18\;9\;10\;6\;1\;15\;17\;3\;7\;12$,\\
			$2\;7\;10\;13\;4\;16\;9\;18\;6\;20\;8\;22\;24\;14\;15\;11\;1\;19\;21\;3\;23\;5\;12\;17$, \\
			$3\;8\;13\;16\;5\;19\;9\;12\;21\;2\;7\;23\;11\;25\;27\;29\;17\;18\;14\;1\;22\;24\;4\;26\;6\;28\;10\;15\;20$, and\\
			$3\;8\;12\;16\;20\;5\;23\;9\;13\;18\;25\;2\;7\;27\;11\;29\;14\;31\;33\;35\;21\;22\;17\;1\;26\;28\;4\;30\;6\;32\;10\;15\;34\;19\;24$,
		\end{tabular}
	\end{center}
respectively. This leads to the following conjecture.

\begin{conjecture}
	Let $\pi$ be a parallel alternation with $|\pi| \geq 6$. Then, $\Av(\pi)$ is deflatable.
\end{conjecture}

There is one parallel alternation (up to symmetry) of length less than $6$: the permutation $2413$. We show in Section~\ref{section:concl} that $\Av(2413)$ is not deflatable.

We conclude this section by generalizing the deflatable class $\Av(251364)$ to an infinite family of deflatable classes. Set $\pi = 251364$ and consider the inflation $\pi^* = \pi[1,\theta, 1, 1, 1, 1]$ for any permutation $\theta$. Set $\omega = 25173486$ (the witness of deflatability for the class $\Av(\pi)$) and further define $\omega^* = \omega[1,\theta,1,\theta,1,1,1,1]$ for the same $\theta$ as before. Both $\pi^*$ and $\omega^*$ are shown in Figure~\ref{figure:pi-omega}.

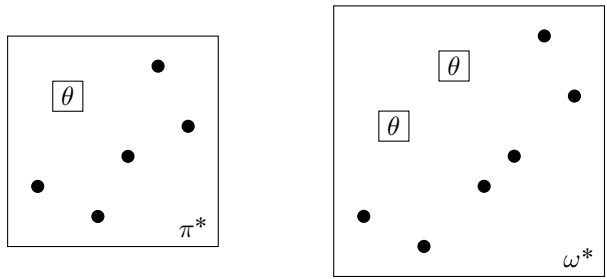
\begin{figure}
	\begin{center}
		\begin{tikzpicture}[scale=.4, baseline=(current bounding box.center)]
			\draw (0,0) rectangle (7,7);
			\draw[black, fill=black] (1,2) circle (0.2);
			\draw (1.5,4.5) rectangle (2.5, 5.5) node[midway] {$\theta$};
			\draw[black, fill=black] (3,1) circle (0.2);
			\draw[black, fill=black] (4,3) circle (0.2);
			\draw[black, fill=black] (5,6) circle (0.2);
			\draw[black, fill=black] (6,4) circle (0.2);
			\node[above left] at (7,0) {$\pi^*$};
		\end{tikzpicture}
		\qquad\qquad
		\begin{tikzpicture}[scale=.4, baseline=(current bounding box.center)]
			\draw (0,0) rectangle (9,9);
			\draw[black, fill=black] (1,2) circle (0.2);
			\draw (1.5,4.5) rectangle (2.5, 5.5) node[midway] {$\theta$};
			\draw[black, fill=black] (3,1) circle (0.2);
			\draw (3.5,6.5) rectangle (4.5, 7.5) node[midway] {$\theta$};
			\draw[black, fill=black] (5,3) circle (0.2);
			\draw[black, fill=black] (6,4) circle (0.2);
			\draw[black, fill=black] (7,8) circle (0.2);
			\draw[black, fill=black] (8,6) circle (0.2);
			\node[above left] at (9,0) {$\omega^*$};
		\end{tikzpicture}
		\caption{The permutations $\pi^*$ (on the left) and $\omega^*$ (on the right).}
	\label{figure:pi-omega}
	\end{center}
\end{figure}

It is fairly straight-forward to see that $\omega^* \in \Av(\pi^*)$, and it is routine to check that any one-point extension of $\omega^*$ by an entry $x$ which splits the interval formed by the entries $3$ and $4$ (without becoming a part of this interval) contains $\pi^*$. Hence, $\Av(\pi^*)$ is deflatable, proving the following theorem.

\begin{theorem}
	There are infinitely many deflatable principal classes.
\end{theorem}

\section{Open Questions}
\label{section:concl}

Although we have shown that there are both infinitely many deflatable principal classes and infinitely many non-deflatable principal classes, the task of classifying exactly which principal classes are deflatable, to say nothing of non-principal classes, remains unfinished. The theorems proved in Section~\ref{section:notdef} combine to prove the non-deflatability of all classes $\Av(\pi)$ for $|\pi| = 4$ (up to symmetry) with the exception of $\Av(2413)$, a special case which we now prove.

\begin{proposition}
	\label{prop:2413}
	The principal class $\Av(2413)$ is not deflatable.
\end{proposition}
\begin{proof}
	Let $\omega \in \Av(2413)$ be indecomposable and not simple. Let $\alpha$ be a longest maximal interval.
	
	We would like to make the assumption that the entry immediately following $\alpha$ by position has value greater than all entries in $\alpha$. Since $2413$ is invariant under all rotations, we can consider an appropriate rotation of $\omega$ so that this is true, for if all rotations had the property that the entry immediately following $\alpha$ in position had value less than all entries of $\alpha$, it would follow that $\omega$ contained an occurrence of $2413$ formed by these four entries (one for each rotation).
	
	Thus, we can assume without loss of generality that the entry immediately following $\alpha$ by position, which we denote by $d$, has value greater than all entries in $\alpha$. Define $d'$ to be the rightmost entry of $\omega$ which separates $\alpha$ from $d$. If there is no such entry, set $d' = d$. Let $a$ be the bottommost entry of $\alpha$. 
	
	Form $\omega^+$ by inserting an entry $x$ into $\omega$ that lies just above $a$ and just to the right of $d$, as in Figure~\ref{figure:2413-1}. We need to show that $\omega^+ \in \Av(2413)$, so suppose toward a contradiction that the entry $x$ plays a role in an occurrence of $\pi$.

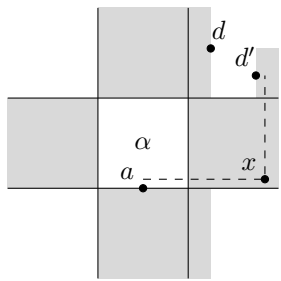
\begin{figure}
	\begin{center}
		\begin{tikzpicture}[scale = .6]
			\draw[draw=none, fill=black!15] (2,0) rectangle ++(2,2);
			\draw[draw=none, fill=black!15] (0,2) rectangle ++(2,2);
			\draw[draw=none, fill=black!15] (2,4) rectangle ++(2,2);
			\draw[draw=none, fill=black!15] (4,2) rectangle ++(2,2);
			\draw[draw=none, fill=black!15] (4,4) rectangle ++(.5,2);
			\draw[draw=none, fill=black!15] (4,0) rectangle ++(.5,2);
			\draw[draw=none, fill=black!15] (5.5,4) rectangle ++(.5,1.1);
			\foreach \i in {2,4} {
				\draw (\i,0) -- (\i,6);
				\draw (0,\i) -- (6,\i);
			}
			\draw[fill=black] (5.7, 2.2) circle (.08) node[left=6pt, above] {$x$};
			\draw[fill=black] (4.5, 5.1) circle (.08) node[right=3pt, above] {$d$};
			\draw[fill=black] (3, 2) circle (.08) node[left=6pt, above] {$a$};
			\draw[fill=black] (5.5, 4.5) circle (.08) node[left=4pt, above] {$d'$};
			\draw[dashed] (3,2.2) -- (5.7,2.2) -- (5.7, 4.5);
			\node at (3,3) {$\alpha$};
		\end{tikzpicture}
		\caption{The permutation diagram of $\omega^+$ in Proposition~\ref{prop:2413}.}
		\label{figure:2413-1}
	\end{center}
\end{figure}
	
	If $x$ played the role of the $2$ in an occurrence of $2413$, then $a$ substitutes for $x$. If $x$ played the role of the $4$ in an occurrence of $2413$, then $d'$ substitutes for $x$. If $x$ plays the role of the $3$ in an occurrence of $2413$, then the $1$ must lie to the left of $\alpha$ (otherwise it acts as the $1$ in an occurrence of $2413$ using $a$, $d$, and $d'$), and so $a$ substitutes for $x$. 
	
	Therefore, $x$ must play the role of the $1$ in some occurrence of $2413$. The entry $d$ cannot play the role of the $4$, because then there are no entries that can play the role of the $3$. The role of $4$ also cannot be played by any entry to the left of $\alpha$, because then $a$ could substitute for $x$. Therefore, the role of $4$ must be played by an entry, say $y$, above $d$ and positionally between $d$ and $d'$. However, this would imply that the role of $3$ was played by an entry, say $z$ above $d$ and to the right of $x$. This in turn implies that $d$, $y$, $d'$, and $z$ form a copy of $2413$, a contradiction.	
\end{proof}

Of the classes $\Av(\pi)$ for $|\pi| = 5$, we have shown that $\Av(\pi)$ is not deflatable for all decomposable $\pi$. The remaining classes $\Av(\pi)$ to be checked are $\Av(25314)$, $\Av(24153)$, $\Av(23514)$, and $\Av(24513)$. Note that the former two bases consist of simple permutations while the latter two consist of inflations of $2413$.  This raises the following question: 

\begin{question}
	\label{question:len-5-def}
	Are the classes $\Av(25314)$, $\Av(24153)$, $\Av(23514)$, and $\Av(24513)$ deflatable?
\end{question}

It is already known that $\pi = 134652$ is a minimal length decomposable $\pi$ such that $\Av(\pi)$ is deflatable. The resolution to Question~\ref{question:len-5-def} would determine whether or not $\pi$ is a minimal length such $\pi$ among all permutations. Moreover, there are three other length $6$ decomposable permutations $\pi$ (up to symmetry) such that the deflatability of $\Av(\pi)$ is unknown. The answer to the Question~\ref{question:len-6-def} might be helpful in determining a more broad classification of deflatable and non-deflatable classes.

\begin{question}
	\label{question:len-6-def}
	Are the classes $\Av(146523)$, $\Av(154623)$, and $\Av(164532)$ deflatable?
\end{question}

The reader may have noticed that, despite proving that many principal classes contain simple permutations which are actually contained in a proper subclass, we have not once specified what that proper subclass is. This is not out of neglect; rather, the only way currently known to calculate the smallest proper subclass $\D \subset \C$ for which $ \C  \subseteq \langle \D \rangle$ is by direct calculation. For the time being, computational power is not sufficient to perform this calculation for the classes in question.

\bigskip

\noindent{\bf Acknowledgments:} The authors are grateful to Vince Vatter for participating in discussions which furthered this research. In particular he was in part responsible for the original proof of Proposition \ref{prop:2413} which convinced us that ``except in trivial cases principal classes aren't deflatable'' was perhaps not as obvious or as easy as one might initially think -- and indeed of course we now know it to be false. Additionally, Cheyne Homberger and Jay Pantone wish to thank Michael Albert and Mike Atkinson for their hospitality at the University of Otago in March and April of 2014.
\bibliographystyle{acm}

\bibliography{../../../../refs/library}

\end{document}